\RequirePackage{rotating}
\documentclass[12pt]{article}

\usepackage[margin=1in]{geometry}

\usepackage{amsmath}
\usepackage{amssymb}
\usepackage{amsthm}
\usepackage{url}
\usepackage{MnSymbol}
\usepackage{graphicx}

\usepackage{tikz} 
\usetikzlibrary{automata,arrows,shapes,topaths,trees,backgrounds,decorations.pathreplacing,positioning}

\usepackage{hyperref}

\usepackage[all,cmtip]{xy}
\usepackage{varwidth}

\newtheorem{theorem}{Theorem}[section]
\newtheorem{lemma}[theorem]{Lemma}

\theoremstyle{definition}

\newtheorem{question}[theorem]{Question}

\newcounter{cases}[theorem]
\newtheorem{case}[cases]{Case}

\numberwithin{subcase}{cases}

\theoremstyle{remark}
\newtheorem{remark}[theorem]{Remark}

\newcommand{\mc}[1]{\mathcal{#1}}

\newcommand{\res}{\!\!\upharpoonright}
\newcommand{\la}{\langle}
\newcommand{\ra}{\rangle}

\newcommand{\treeres}[2]{\{\rotatebox[origin=c]{90}{$\rightsquigarrow$}_{#1}^{#2}\}}

\newcommand{\Labels}{\mathtt{Labels}}

\DeclareMathOperator{\scope}{scope}
\DeclareMathOperator{\pred}{pred}

\newcommand{\ptime}{\mathsf{P}}
\newcommand{\nptime}{\mathsf{NP}}

\newcommand{\comp}{\mathsf{COMP}}
\newcommand{\pspace}{\mathsf{PSPACE}}

\newcommand{\concat}[0]{\textrm{\^{}}}

\begin{document}

\title{A Minimal Set Low for Speed}
\author{Rod Downey \and Matthew Harrison-Trainor}

\makeatletter
\def\blfootnote{\xdef\@thefnmark{}\@footnotetext}
\makeatother

\maketitle

\begin{abstract}
	An oracle $A$ is low-for-speed if it is unable to speed up the computation of a set which is already computable: if a decidable language can be decided in time $t(n)$ using $A$ as an oracle, then it can be decided without an oracle in time $p(t(n))$ for some polynomial $p$. The existence of a set which is low-for-speed was first shown by Bayer and Slaman who constructed a non-computable computably enumerable set which is low-for-speed. In this paper we answer a question previously raised by Bienvenu and Downey, who asked whether there is a minimal degree which is low-for-speed. The standard method of constructing a set of minimal degree via forcing is incompatible with making the set low-for-speed; but we are able to use an interesting new combination of forcing and full approximation to construct a set which is both of minimal degree and low-for-speed.
\end{abstract}

\section{Introduction}

Almost since the beginning of computational complexity theory,
we have had results about oracles and their effect 
on the running times of computations. 
For example Baker, Gill, and Solovay \cite{BakerGS1975} showed that on the one hand there are oracles~$A$ such that $\ptime^A=\nptime^A$ and on the other hand there are oracles~$B$ such that $\ptime^B \not= \nptime^B$, thus demonstrating that methods that relativize will not suffice to solve basic questions like $\ptime$ vs $\nptime$.
An underlying question is whether oracle results can say things about complexity questions in the unrelativized world.
The answer seems to be yes. For example,
Allender together with Buhrman and Kouck\'{y} \cite{AllenderBK2006} and with Friedman and Gasarch
\cite{AllenderFG2013} showed that 
oracle access to sets of random strings 
can give insight into 
basic complexity questions. 
In~\cite{AllenderFG2013} Allender, Friedman, and Gasarch
showed that 
$\bigcap_U \ptime^{R_{K_U}}\cap \comp\subseteq \pspace$
where $R_{K_U}$ denotes the strings whose prefix-free Kolmogorov complexity
(relative to universal machine $U$) is at least their length, and 
$\comp$ denotes the collection of computable sets.
Later the ``$\cap \comp$'' was removed by 
Cai, Downey, Epstein, Lempp, and J. Miller \cite{CaiDELM2014}.
Thus we conclude that reductions to very complex sets like the 
random strings somehow gives insight into 
very simple things like computable sets.

One of the classical notions in computability theory is that of lowness.
An oracle is low for a specific type of problem if that oracle does not help to solve that problem. A language $A$
is low if the halting problem relative to $A$ has the same Turing degree
(and hence the same computational content) as the halting problem.
Slaman and Solovay \cite{SS} characterized languages $L$ 
where oracles are of no help in Gold-style learning theory:
$EX^L=EX$ iff $L$ is low and 1-generic.
Inspired by this and other lowness results in classical computability,
Allender asked whether 
there were non-trivial sets which were ``low for speed'' in that, 
as oracles, they did not accelerate 
running times  
of computations by more than a polynomial amount.
Of course, as stated this makes little sense since 
using any $X$ as oracle, we can decide membership in~$X$ in linear time, while without an oracle $X$ may not even be computable at all!
Thus, what we are really interested in is the set of oracles which do not speed up the computation of \emph{computable sets} by more than a polynomial amount. More precisely, an oracle~$X$ is \emph{low for speed} if for any computable language~$L$, if some Turing machine $M$ with access to oracle~$X$ decides $L$ in time~$f$, then there is a Turing machine~$M'$ without any oracle and a polynomial~$p$ such that $M'$ decides~$\mathcal{L}$ in time~$p \circ f$. (Here the computation time of an oracle computation is counted in the usual complexity-theoretic fashion: we have a query tape on which we can write strings, and once a string~$x$ is written on this tape, we get to ask the oracle whether $x$ belongs to it in time $O(1)$.)

There are trivial examples of such sets, namely oracles that belong to~$\ptime$, because any query to such an oracle can be replaced by a polynomial-time computation. Allender's precise question was therefore: 

\begin{center}
	Is there an oracle $X \notin \ptime$ which is low for speed?
\end{center}

\noindent Such an~$X$, if it exists, has to be non-computable, for the same reason as above: if $X$ is computable and low for speed, then $X$ is decidable in linear time using oracle~$X$, thus---by lowness---decidable in polynomial time without oracle, i.e., $X \in \ptime$. 

A partial answer was given by Lance Fortnow (unpublished), who observed the following.

\begin{theorem}[Fortnow] If $X$ is a hypersimple and computably enumerable oracle, then $X$ is low for polynomial time, in that if $L \in \ptime^X$ is computable, then $L \in \ptime$.
\end{theorem}

Allender's question was finally solved by Bayer and Slaman, who showed the following. 

\begin{theorem}[Bayer-Slaman \cite{Bayer-PhD}] 
	There are non-computable, computably enumerable, sets~$X$ which are low for speed.  
\end{theorem} 

\noindent Bayer showed that whether 1-generic sets were low for speed depended on
whether $\ptime = \nptime$. 
In \cite{BienvenuDowney}, Bienvenu and Downey began an analysis of precisely 
what kind of sets/languages could be low for speed. The showed for instance,
randomness always accelerates some computation in that no 
Schnorr random set is low for speed. 
They also constructed a perfect $\Pi_1^0$ class all of whose members were low for speed. Among other results, they demonstrated that being low for speed 
did not seem to align very well to having low complexity in that 
no set of low computably enumerable Turing degree could also be low for speed.

From one point of view the sets with barely non-computable information are 
those of minimal Turing degree. Here we recall that ${\bf a}$ is a 
\emph{minimal} Turing degree if it is nonzero and there is 
no degree ${\bf b} $ with ${\bf 0}<{\bf b}<{\bf a}$.
It is quite easy to construct a set of minimal Turing degree which is 
not low for speed, and indeed any natural minimality construction seems to give this.
That is because natural Spector-forcing style dynamics  seem 
to entail certain delays in any construction, even a full approximation one,
which cause problems with the polynomial time simulation of 
the oracle computations being emulated.
In view of this, Bienvenu and Downey
asked the following question:

\begin{question} Can a set $A$ of minimal Turing degree be low for speed?
\end{question}

In the present paper we answer this question affirmatively:

\begin{theorem} There is a set $A$ which is both of minimal Turing degree and
	low for speed.
\end{theorem}

\noindent The construction is a mix of forcing and full approximation of a kind 
hitherto unseen. The argument is a complicated priority construction in which the interactions between different requirements is quite involved. In order to make the set $A$ of minimal Turing degree, we must put it on splitting trees; and in order to make it low-for-speed, we must have efficient simulations of potential computations involving $A$. When defining the splitting trees, we must respect decisions made by our simulations, which restricts the splits we can choose. The splitting trees end up having the property that while every two paths through the tree split, two children of the same node may not split; finding splits is sometimes very delayed. This is a new strategy which does not seem to have been used before for constructing sets of minimal degree.

\section{The construction with few requirements}\label{sec:two}

We will construct a set $A$ meeting the following requirements:
\begin{align*}
\mathcal{M}_e:\qquad& \text{If $\Phi_e^A$ is total then it is either computable or computes $A$.}\\
\mathcal{L}_{\la e,i \ra}:\qquad& \text{If $\Psi_{e}^A = R_{i}$ is total and computable in time $t(n)$, then it is}\\ &\text{ computable in time $p(t(n))$ for some polynomial $p$.}\\
\mathcal{P}_e:\qquad& \text{$A \neq W_e$.}
\end{align*}
Here, $R_{i}$ is a partial computable function. The requirements $\mathcal{P}_e$ make $A$ non-computable, while the requirements $\mathcal{M}_e$ make $A$ of minimal degree. The requirements $\mathcal{L}_{\la e,i \ra}$ make $A$ low for speed.

\medskip{}

When working with Spector-style forcing, it is common to define a tree to be a map $T \colon 2^{< \omega} \to 2^{< \omega}$ such that $\sigma \preceq \tau$ implies $T(\sigma) \preceq T(\tau)$. We will need our trees to be finitely branching; so for the purposes of this proof a tree will be a computable subset of $2^{< \omega}$ so that each node $\sigma$ on the tree has finitely many children $\tau \succeq \sigma$. The children of $\sigma$ may be of any length, where by length we mean the length as a binary string. Our trees will have no dead ends, and in fact every node will have at least two children.
As usual, $[T]$ denotes the collection of paths through $T$.
Recall that for a functional $\Phi_e$, we 
will say that $T$ is $e$-\textit{splitting} if for any two distinct paths $\pi_1$ and $\pi_2$ through $T$, 
there is $x$ with 
\[\Phi_e^{\pi_1}(x)\downarrow \ne \Phi_e^{\pi_2}(x)\downarrow.\]
If $\tau_1$ and $\tau_2$ are initial segments of $\pi_1$ or $\pi_2$ respectively witnessing this, i.e., with
\[\Phi_e^{\tau_1}(x)\downarrow \ne \Phi_e^{\tau_2}(x)\downarrow,\]
and with a common predecessor $\sigma$, we say that they $e$-split over $\sigma$, or that they are an $e$-split over $\sigma$.
The requirements $\mathcal{M}_e$ will be satisfied by an interesting new mix of forcing and full approximation. Following the standard Spector argument, to satisfy $\mathcal{M}_e$ we attempt to make $A$ a path on a tree $T$ with either:
\begin{itemize}
	\item $T$ is $e$-splitting, and so for any path $B \in [T]$, $\Phi_e^B \geq_T B$; or
	\item for all paths $B_1,B_2 \in [T]$ and all $x$, if $\Phi_e^{B_1}(x) \downarrow$ and $\Phi_e^{B_2}(x) \downarrow$ then $\Phi_e^{B_1}(x) = \Phi_e^{B_2}(x)$, and so $\Phi_e^B$ is either  partial or computable for any $B \in [T]$.
\end{itemize}
Given such a tree, any path on $T$ satisfies $\mathcal{M}_e$.

The standard argument for building a minimal degree is a forcing argument. Suppose that we want to meet just the $\mathcal{M}$ and $\mathcal{P}$ requirements. We can begin with a perfect tree $T_{-1}$, say $T_{-1} = 2^{<\omega}$. Then there is a computable tree $T_0 \subseteq T_{-1}$ which is either $0$-splitting or forces $\Phi_0^A$ to be either computable or partial. We can then choose $A_0 \in T_0$ such that $A_0$ is not an initial segment of $W_e$. Then there is a computable tree $T_1 \subseteq T_0$ with root $A_0$ which is either $1$-splitting or forces $\Phi_1^A$ to be either computable or partial. We pick $A_1 \in T_1$ so that $A_1$ is not an initial segment of $W_1$, then $T_2 \subseteq T_1$ with root $A_1$, and so on. Then $A = \bigcup A_i$ is a path through each $T_i$, and so is a minimal degree. Though each $T_i$ is computable, they are not uniformly computable; given $T_i$, to compute $T_{i+1}$ we must know whether $T_{i+1}$ is to be $(i+1)$-splitting, to force partiality, or to force computability.

\medskip{}

We cannot purely use forcing the meet the lowness requirements $\mathcal{L}_{\la e,i\ra}$. We use something similar to the Slaman-Beyer strategy from \cite{Bayer-PhD}. The entire construction will take place on a tree $T_{-1}$ with the property that it is polynomial in $|\sigma|$ to determine whether $\sigma \in T_{-1}$, and that moreover, for each $n$, there are polynomially many in $n$ strings of length $n$ on $T_{-1}$. For example, let $\sigma \in T_{-1}$ if it is of the form
\[ a_1^{2^0} a_2^{2^1} a_3^{2^2} a_4^{2^3} \cdots \]
where each $a_i \in \{0,1\}$.
So, for example, $100111100000000 \in T_{-1}$.

First we will show how to meet $\mathcal{L}_{\la e,i \ra}$ in the absence of any other requirements. For simplicity drop the subscripts $e$ and $i$ so that we write $\Psi = \Psi_{e}$ and $R = R_{i}$. The idea is to construct a computable simulation $\Xi$ of $\Psi^A$, with $\Xi(x)$ computable in time polynomial in the running time of $\Psi^A(x)$, so that if $\Psi^A = R$ then $\Xi = \Psi^A$. We compute $\Xi(x)$ as follows. We computably search over $\sigma \in T_{-1}$ (i.e.\ over find potential initial segments $\sigma$ of $A$) and simulate the computations $\Psi^\sigma(x)$. When we find $\sigma$ with $\Psi^\sigma(x) \downarrow$, we set $\Xi(x) = \Psi^\sigma(x)$ for the first such $\sigma$. Of course, $\sigma$ might not be an initial segment of $A$, and so $\Xi$ might not be equal to $\Psi^A$; this only matters if $\Psi^A = R$ is total, as otherwise $\mathcal{L}$ is satisfied vacuously. If $x$ is such that $\Xi(x) \downarrow \neq R(x) \downarrow$, then there is some $\sigma \in T_{-1}$ witnessing that $\Psi^\sigma(x) = \Xi(x)$; the requirement $\mathcal{L}$ asks that $A$ extend $\sigma$, so that $\Psi^A(x) \neq R(x)$ and $\mathcal{L}$ is satisfied. So now we need to ensure that if $\Xi = \Psi^A = R$, then $\Xi$ is only polynomially slower than $\Psi^A$. We can do this by appropriately dovetailing the simulations so that if $\Psi^\sigma(x) \downarrow$ in time $t(x)$, the simulation $\Xi$ will test this computation in a time which is only polynomially slower than $t(x)$, and we will have $\Xi(x) \downarrow$ in time which is only polynomially slower than $t(x)$. For example, we might start by simulating one stage of the computation $\Psi^\sigma(x)$ for $\sigma$s of length one, then simulating two stages for $\sigma$s of length at most two, then three stages for $\sigma$s of length at most three, and so on. It is important here that $T_{-1}$ has only polynomially many nodes at height $n$ and we can test membership in $T_{-1}$ in polynomial time; so the $n$th round of simulations takes time polynomial in $n$.

Think of the simulations as being greedy and taking any computation that they find; and then, at the end, we can non-uniformly choose the initial segment of $A$ to force that either the simulation is actually correct, or to get a diagonalization.

\medskip{}

The interactions between the requirements get more complicated. Consider now two requirements $\mathcal{M} = \mc{M}_e$ and $\mathcal{L}$. If $\mc{L}$ is of higher priority than $\mc{M}$, there is nothing new going on---$\mc{M}$ knows whether $\mc{L}$ asked to have $A$ extend some node $\sigma$, and if it did, $\mc{M}$ tries to build a splitting tree extending $\sigma$. So assume that $\mathcal{M}$ is of higher priority than $\mathcal{L}$.

Write $\Phi = \Phi_e$. Assume that for each $\sigma \in T_{-1}$, there are $x$ and $\tau_1,\tau_2 \succeq \sigma$ such that $\Phi^{\tau_1}(x) \downarrow \neq \Phi^{\tau_2}(x) \downarrow$; and that for each $\sigma \in T_{-1}$ and $x$ there is $\tau \succeq \sigma$ such that $\Phi^{\tau}(x) \downarrow$. Otherwise, we could find a subtree of $T_{-1}$ which forces that $\Phi^A$ is either not total or is computable, and satisfy $\mc{M}$ by restricting to that subtree. This assumption implies that we can also find any finite number of extensions of various nodes that pairwise $e$-split, e.g. given $\sigma_1$ and $\sigma_2$, there are extensions of $\sigma_1$ and $\sigma_2$ that $e$-split. Indeed, find extensions $\tau,\tau^*$ of $\sigma_1$ that $e$-split, say $\Phi^{\tau_1}(x) \downarrow \neq \Phi^{\tau_2}(x) \downarrow$, and an extension $\rho$ of $\sigma_2$ with $\Phi^{\tau}(x) \downarrow$. Then $\rho$ $e$-splits with one of $\tau_1$ or $\tau_2$.

The requirement $\mathcal{L}$ non-uniformly guesses at whether or not $\mathcal{M}$ will succeed at building an $e$-splitting tree. Suppose that it guesses that $\mathcal{M}$ successfully builds such a tree. $\mathcal{M}$ begins with the special tree $T_{-1}$ described above, and it must build an $e$-splitting tree $T \subseteq T_{-1}$.

While building the tree, $\Xi$ will be simulating $\Psi^A$ by looking at computations $\Psi^\sigma$. The tree $T$ might be built very slowly, while $\Xi$ has to simulate computations relatively quickly. So when a node is removed from $T$, $\Xi$ will stop simulating it, but $\Xi$ will have to simulate nodes which are extensions of nodes in $T$ as it has been defined so far, but which have not yet been determined to be in or not in $T$. This leads to the following problem: Suppose that $\gamma$ is a leaf of $T$ at stage $s$, $\rho$ extends $\gamma$, and $\Xi$ simulates $\Psi^\rho(x)$ and sees that it converges, and so defines $\Xi(x) = \Psi^\rho(x)$. But then the requirement $\mathcal{M}$ finds an $e$-split $\tau_1,\tau_2 \succeq \gamma$ and wants to set $\tau_1$ and $\tau_2$ to be the successors of $\gamma$ on $T$, with both $\tau_1$ and $\tau_2$ incompatible with $\rho$. If we allow $\mathcal{M}$ to do this, then since $\mathcal{M}$ has higher priority than $\mathcal{L}$, $\mathcal{M}$ has determined that $A$ cannot extend $\rho$ as $\mc{M}$ restricts $A$ to be a path through $T$. So $\mc{L}$ has lost its ability to diagonalize and it might be that $\Psi^A = R$ (say, because this happens on all paths through $T$) but $\Psi^A(x) \neq \Psi^\rho(x) = \Xi(x)$. 

This means that $\mathcal{M}$ needs to take some action to keep computations that $\mathcal{L}$ has found on the tree. We begin by describing the most basic strategy for keeping a single node $\rho$ on the tree.

Suppose that at stage $s$ the requirement $\mathcal{M}$ wants to add children to a leaf node $\gamma$ on $T$. First, look for $\gamma_1,\gamma_2,\gamma_3$ extending $\gamma$ such that they pairwise $e$-split: for any two $\gamma_i,\gamma_j$, there is $x$ such that $\Phi_e^{\gamma_i}(x) \downarrow \neq \Phi_e^{\gamma_j}(x) \downarrow$. By our earlier assumption that each node on $T_{-1}$ has an $e$-splitting extension we will eventually find such elements, say at stage $t$. But it might be that by stage $t$, we have simulated $\Psi_t^\rho(x) \downarrow$ and set $\Xi(x)$ to be equal to this simulated computation. So for the sake of $\mc{L}$, we must keep $\rho$ on the tree. (Later, we will have to extend the strategy to worry about what happens if we have simulated multiple computations $\rho$, but for now assume that there is just one.)

To begin, we stop simulating any computations extending $\rho$. This means that we are now free to extend the tree however we like above $\rho$ without worrying about how this affects the simulations. We also stop simulating any other computations not compatible with $\gamma_1$, $\gamma_2$, or $\gamma_3$.

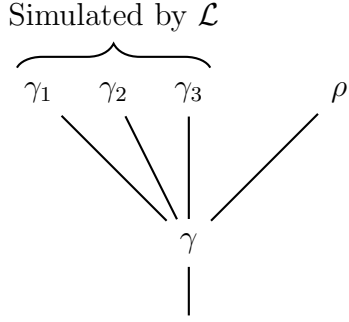
\begin{center}
	\begin{tikzpicture}[-,>=stealth',shorten >=1pt,shorten <=1pt, auto,node
	distance=2cm,thick,every loop/.style={<-,shorten <=1pt}]

	\node (gamma) at (0,1) {{$\gamma$}};
	
	\node (gamma1) at (-2,3) {{$\gamma_1$}};
	\node (gamma2) at (-1,3) {{$\gamma_2$}};
	\node (gamma3) at (0,3) {{$\gamma_3$}};
	\node (rho) at (2,3) {{$\rho$}};
	
	\path (0,0) edge[-] (gamma);
%
	\path (gamma) edge[-] (gamma1);
	\path (gamma) edge[-] (gamma2);
	\path (gamma) edge[-] (gamma3);
	\path (gamma) edge[-] (rho);
%
%
	\node (rhoe) at (-1,4) {{Simulated by $\mc{L}$}};
%
	\draw [decorate,decoration={brace,amplitude=10pt}]
	(-2.3,3.3) -- (0.3,3.3);
	\end{tikzpicture}
\end{center}

Now look for an extension $\rho^*$ of $\rho$ that $e$-splits with at least two of $\gamma_1$, $\gamma_2$, and $\gamma_3$. We can find such a $\rho^*$ by looking for one with $\Phi^{\rho^*}$ defined on the values $x$ witnessing the $e$-splitting of  $\gamma_1$, $\gamma_2$, and $\gamma_3$, e.g., if $\Phi^{\gamma_1}(x) \neq \Phi^{\gamma_2}(x)$, and $\Phi^{\rho^*}(x) \downarrow$, then $\rho^*$ must $e$-split with either $\gamma_1$ or $\gamma_2$. Say that $\rho^*$ $e$-splits with $\gamma_1$ and $\gamma_2$. Then we define the children of $\gamma$ to be $\gamma_1$, $\gamma_2$, and $\rho^*$.

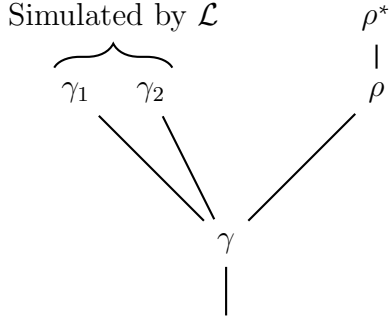
\begin{center}
	\begin{tikzpicture}[-,>=stealth',shorten >=1pt,shorten <=1pt, auto,node
	distance=2cm,thick,every loop/.style={<-,shorten <=1pt}]

	\node (gamma) at (0,1) {{$\gamma$}};
	
	\node (gamma1) at (-2,3) {{$\gamma_1$}};
	\node (gamma2) at (-1,3) {{$\gamma_2$}};
	\node (rho) at (2,3) {{$\rho$}};
	\node (rhos) at (2,4) {{$\rho^*$}};
	
	\path (0,0) edge[-] (gamma);
	%
	\path (gamma) edge[-] (gamma1);
	\path (gamma) edge[-] (gamma2);
	\path (gamma) edge[-] (rho);
	\path (rho) edge[-] (rhos);
	%
	%
	\node (rhoe) at (-1.5,4) {{Simulated by $\mc{L}$}};
	%
	\draw [decorate,decoration={brace,amplitude=10pt}]
	(-2.3,3.3) -- (-0.7,3.3);
	\end{tikzpicture}
\end{center}
\noindent So of the extensions of $\gamma$, some are simulated by $\mc{L}$, and others are not. If $\mc{L}$ has the infinitary outcome, where it never finds the need to diagonalize, then it will have $A$ extend either $\gamma_1$ or $\gamma_2$. It is only if $\mc{L}$ needs to diagonalize that it will have $A$ extend $\rho^*$---and in this case, $\mc{L}$ is satisfied and so does not have to simulate $\Psi^A$.

There is still an issue here. What if, while looking for $\rho^*$, we simulate a computation $\Psi^{\gamma_3}(y) \downarrow$, and set $\Xi(y) = \Psi^{\gamma_3}(y)$, and then only after this find that $\rho^*$ $e$-splits with $\gamma_1$ and $\gamma_2$? We can no longer remove $\gamma_3$ from the tree. Moreover, there might be many different nodes $\rho$ that we cannot remove from the tree---indeed, it might be that around stage $s$, we cannot remove any nodes at height $s$ from the tree, because each of them has some computation that we have simulated.

To deal with this, we have to build $e$-splitting trees in a weaker way. It will no longer be the case that every pair of children of a node $\sigma$ $e$-split, but we will still make sure that every pair of paths $e$-splits. (It might seem that this violates compactness, but in fact thinking more carefully it does not---the set of pairs of paths $(\pi_1,\pi_2)$ that $e$-split is an open cover of the non-compact topological space $[T]\times [T] - \Delta$ where $\Delta$ is the diagonal.)

So suppose again that we are trying to extend $\gamma$. Look for a pair of nodes $\gamma_1,\gamma_2$ that $e$-split. Suppose that $\rho_1,\ldots,\rho_n$ are nodes that have been simulated, that we must keep on the tree. (We might even just assume that $\rho_1,\ldots,\rho_n$ are all of the other nodes at the same level as $\gamma_1,\gamma_2$.) We stop simulating computations above $\rho_1,\ldots,\rho_n$.
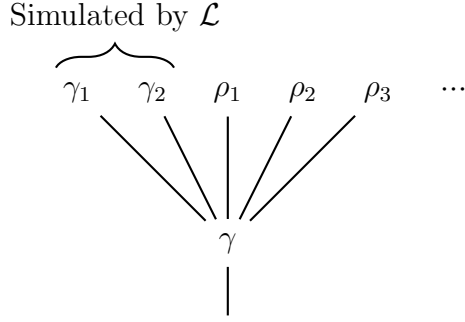
\begin{center}
	\begin{tikzpicture}[-,>=stealth',shorten >=1pt,shorten <=1pt, auto,node
	distance=2cm,thick,every loop/.style={<-,shorten <=1pt}]

	\node (gamma) at (0,1) {{$\gamma$}};
	
	\node (gamma1) at (-2,3) {{$\gamma_1$}};
	\node (gamma2) at (-1,3) {{$\gamma_2$}};
	\node (rho1) at (0,3) {{$\rho_1$}};
	\node (rho2) at (1,3) {{$\rho_2$}};
	\node (rho3) at (2,3) {{$\rho_3$}};
	\node (dots) at (3,3) {{$\cdots$}};
	
	\path (0,0) edge[-] (gamma);
	%
	\path (gamma) edge[-] (gamma1);
	\path (gamma) edge[-] (gamma2);
	\path (gamma) edge[-] (rho1);
	\path (gamma) edge[-] (rho2);
	\path (gamma) edge[-] (rho3);
	%
	%
	\node (rhoe) at (-1.5,4) {{Simulated by $\mc{L}$}};
	%
	\draw [decorate,decoration={brace,amplitude=10pt}]
	(-2.3,3.3) -- (-0.7,3.3);
	\end{tikzpicture}
\end{center}
Now at the next step we need to add extensions to $\gamma_1$ and $\gamma_2$ just as we added extensions of $\gamma$. We look for extensions $\gamma_1^*$ and $\gamma_1^{**}$ of $\gamma_1$, $\gamma_2^*$ and $\gamma_2^{**}$ of $\gamma_2$, and $\rho_i^*$ and $\rho_i^{**}$ of $\rho_i$ such that all of these pairwise $e$-split. While we are looking for these, we might simulate more computations at nodes $\tau$ above $\gamma_1$ and $\gamma_2$, but there will be no more computations simulated above the $\rho_i$.
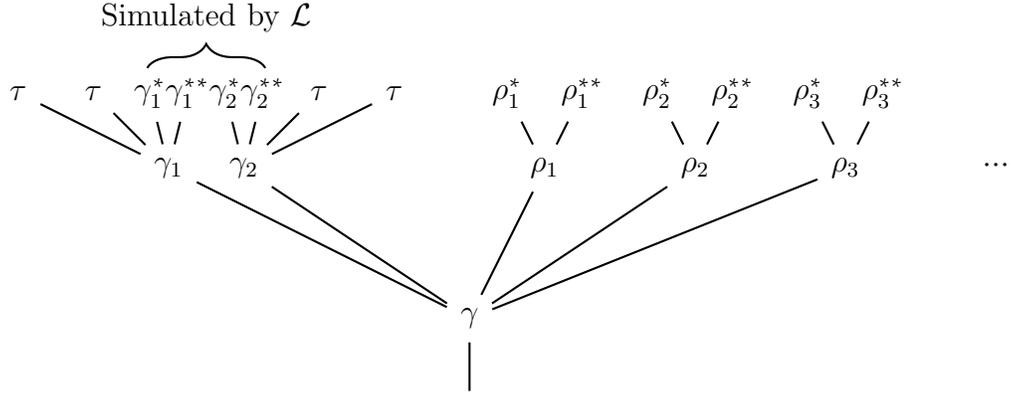
\begin{center}
	\begin{tikzpicture}[-,>=stealth',shorten >=1pt,shorten <=1pt, auto,node
	distance=2cm,thick,every loop/.style={<-,shorten <=1pt}]

	\node (gamma) at (0,1) {{$\gamma$}};
	
	\node (tau1) at (-5,4) {{$\tau$}};
	\node (tau2) at (-6,4) {{$\tau$}};
	\node (tau3) at (-1,4) {{$\tau$}};	
	\node (tau4) at (-2,4) {{$\tau$}};
	
	\node (gamma1) at (-4,3) {{$\gamma_1$}};
	\node (gamma2) at (-3,3) {{$\gamma_2$}};
	\node (gamma1s) at (-4.25,4) {{$\gamma_1^*$}};
	\node (gamma2s) at (-2.75,4) {{$\gamma_2^{**}$}};
	\node (gamma1ss) at (-3.75,4) {{$\gamma_1^{**}$}};
	\node (gamma2ss) at (-3.25,4) {{$\gamma_2^{*}$}};
	\node (rho1) at (1,3) {{$\rho_1$}};
	\node (rho1s) at (0.5,4) {{$\rho_1^*$}};
	\node (rho1ss) at (1.5,4) {{$\rho_1^{**}$}};
	\node (rho2) at (3,3) {{$\rho_2$}};
	\node (rho2s) at (2.5,4) {{$\rho_2^*$}};
	\node (rho2ss) at (3.5,4) {{$\rho_2^{**}$}};
	\node (rho3) at (5,3) {{$\rho_3$}};
	\node (rho3s) at (4.5,4) {{$\rho_3^*$}};
	\node (rho3ss) at (5.5,4) {{$\rho_3^{**}$}};
	\node (dots) at (7,3) {{$\cdots$}};

	\path (0,0) edge[-] (gamma);
	%
	\path (gamma) edge[-] (gamma1);
	\path (gamma) edge[-] (gamma2);
	\path (gamma1) edge[-] (gamma1s);
	\path (gamma1) edge[-] (gamma1ss);
	\path (gamma1) edge[-] (tau1);
	\path (gamma1) edge[-] (tau2);
	\path (gamma2) edge[-] (tau3);
	\path (gamma2) edge[-] (tau4);
	\path (gamma2) edge[-] (gamma2s);
	\path (gamma2) edge[-] (gamma2ss);
	\path (gamma) edge[-] (rho1);
	\path (gamma) edge[-] (rho2);
	\path (gamma) edge[-] (rho3);
	\path (rho1) edge[-] (rho1s);
	\path (rho2) edge[-] (rho2s);
	\path (rho3) edge[-] (rho3s);
	\path (rho1) edge[-] (rho1ss);
	\path (rho2) edge[-] (rho2ss);
	\path (rho3) edge[-] (rho3ss);
	%
	%
	\node (rhoe) at (-3.5,5) {{Simulated by $\mc{L}$}};
	%
	\draw [decorate,decoration={brace,amplitude=10pt}]
	(-4.3,4.3) -- (-2.7,4.3);
	\end{tikzpicture}
\end{center}
Now at the next step of extending the tree we need to extend $\gamma_1^*$, $\gamma_1^{**}$, $\gamma_2^*$, and $\gamma_2^{**}$, and make sure that we extend $\tau$ to $e$-split with these extensions and with extensions of the $\rho^*$; but in doing so we will introduce further extensions that do not $e$-split. So at no finite step do we get that everything $e$-splits with each other, but in the end every pair of paths $e$-splits.

\section{Multiple requirements and outcomes}\label{sec:three}

Order the requirements $\mathcal{M}_e$, $\mathcal{L}_e$, and $\mathcal{P}_e$ as follows, from highest priority to lowest:
\[ \mc M_0 > \mc L_0 > \mc P_0 > \mc M_1 > \mc L_1 > \mc P_1 > \mc M_2 > \cdots .\]
Each requirement has various possible outcomes:
\begin{itemize}
	\item A requirement $\mathcal{M}_e$ can either build an $e$-splitting tree, or it can build a tree forcing that $\Phi_e$ is either partial or computable. In the former case, when $\mathcal{M}_e$ builds an $e$-splitting tree, we say that $\mathcal{M}_e$ has the infinitary outcome $\infty$. In the latter case, there is a node $\sigma$ above which we do not find any more $e$-splittings. We say that $\mathcal{M}_e$ has the finitary outcome $\sigma$.
	\item A requirement $\mathcal{L}_{\la e,i\ra}$ can either have the simulation $\Xi$ of $\Psi_{e}$ be equal to $R_{i}$ whenever they are both defined, or $\mathcal{L}_{\la e,i\ra}$ can force $A$ to extend a node $\sigma$, with $\Psi_e^\sigma(x) \neq R_i(x)$ for some $x$. In the first case, we say that $\mathcal{L}_{\la e,i\ra}$ has the infinitary outcome $\infty$, and in the latter case we say that $\mathcal{L}_{\la e,i\ra}$ has the finitary outcome $\sigma$.
	\item A requirement $\mathcal{P}_e$ chooses an initial segment $\sigma$ of $A$ that ensures that $A$ is not equal to the $e$th c.e.\ set $W_e$. This node $\sigma$ is the outcome of $\mathcal{P}_e$.
\end{itemize}
The \textit{tree of outcomes} is the tree of finite strings $\eta$ where $\eta(3e)$ is an outcome for $\mathcal{M}_e$, $\eta(3e+1)$ is an outcome for $\mathcal{L}_e$, and $\eta(3e+2)$ is an outcome for $\mathcal{P}_e$, and so that $\eta$ satisfies the coherence condition described below. For convenience, given a requirement $\mc{R}$ we write $\eta(\mc{R})$ for the outcome of $\mc{R}$ according to $\eta$: $\eta(\mc{M}_e) = \eta(3e)$, $\eta(\mc{L}_e) = \eta(3e+1)$, and $\eta(\mc{P}_e) = \eta(3e+2)$. Using this notation allows us to avoid having to remember exactly how we have indexed the entries of $\eta$. Given a requirement $\mc{R}$, we say that $\eta$ is a \textit{guess by $\mc{R}$} if $\eta$ has an outcome for each requirement of higher priority than $\mc{R}$, e.g.\ a guess by $\mathcal{L}_e$ is a string $\eta$ of length $3e+1$ with \[ \eta =  \la \eta(\mc{M}_0),\eta(\mc{L}_0),\eta(\mc{P}_0),\ldots,\eta(\mc{M}_{e-1}),\eta(\mc{L}_{e-1}),\eta(\mc{P}_{e-1}),\eta(\mc{M}_e) \ra.\]
Not all of these guesses are internally consistent; the guesses which might actually be the true outcomes will satisfy other conditions, for example the nodes in the guess will actually be on the trees built by higher priority requirements.

Each requirement $\mc{R}$ will have an \textit{instance} $\mc{R}^\eta$ operating under each possible guess $\eta$ at the outcomes of the lower priority requirements. Each of these instances of a particular requirement will be operating independently, but the actions of all of the instances of all the requirements will be uniformly computable. So for example each instance $\mathcal{M}_e^\eta$ of a minimality requirement will be trying to build an $e$-splitting tree, using $\eta$ to guess at whether or not $\mathcal{M}_{e-1}$ successfully built an $e$-splitting tree, how $\mathcal{L}_{e-1}$ was satisfied, and the node chosen by $\mathcal{P}_{e-1}$ as an initial segment of $A$. The instances of different requirements will not be completely independent; for example, a requirement $\mathcal{M}_e^\eta$ must take into account all of the lower priority requirements $\mathcal{L}_d^\nu$ for $d \geq e$, $\nu \succ \eta$.

One can think of the argument as a forcing argument except that there are some (effective) interactions between the conditions. In a standard forcing construction to build a minimal degree, for each requirement $\mathcal{M}_e$, after forcing the outcomes of the lower priority requirements, we decide non-uniformly whether we can find an $e$-splitting tree $T_e$, or whether there is a node $\sigma$ which has no $e$-splitting tree extending it. The tree $T_e$ is in some sense built after deciding on the outcomes of the previous requirements. What we will do is attempt to build, for each guess $\eta$ of $\mathcal{M}_e$ at the outcomes of the lower priority requirements, an $e$-splitting tree $T_e^\eta$; and then we will, at the end of the construction, choose one instance $\mc{M}_e^\eta$ of $\mathcal{M}_e$ to use depending on the outcomes of the lower priority requirements, and then we use the tree $T_e^\eta$ built by that instance. All of these trees $T_e^\eta$ were already built before we started determining the outcomes of the requirements.

There is only one instance $\mathcal{M}_0^\varnothing$ of $\mathcal{M}_0$, since there are no higher priority requirements. After the construction, we will ask $\mc{M}^\varnothing_0$ what its outcome was. We then have an instance of $\mathcal{L}_0$ which guessed this outcome for $\mathcal{M}_0$, and we ask this instance what its outcome was. This gives us an instance of $\mathcal{P}_0$ that guessed correctly, and so on. So at the end, we use only one instance of each requirement, and follow whatever that instance did.

\bigskip{}

We now need to consider in more detail the interactions between the requirements. We saw in the previous section that an $\mathcal{M}$ requirement must take into account lower priority $\mathcal{L}$ requirements. In the full construction, we will have not only many different lower priority $\mathcal{L}$ requirements, but also many different instances of each one that the $\mathcal{M}$ requirement must take into account.

Consider three requirements, $\mc{M}$ of highest priority, $\mc{L}_0$ of middle priority, and $\mc{L}_1$ of lowest priority. Suppose that both $\mc{L}_0$ and $\mc{L}_1$ correctly guess that $\mc{M}$ has the infinitary outcome, building a splitting tree. As described before, when we extend $\gamma$, we get a picture as follows (ignoring $\mc{L}_1$ for now):
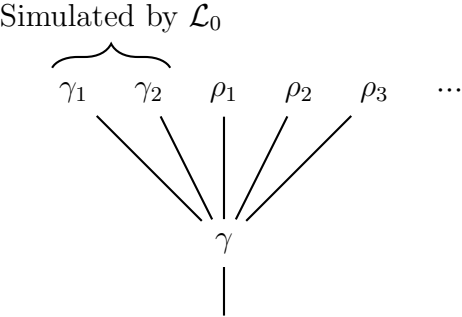
\begin{center}
	\begin{tikzpicture}[-,>=stealth',shorten >=1pt,shorten <=1pt, auto,node
	distance=2cm,thick,every loop/.style={<-,shorten <=1pt}]

	\node (gamma) at (0,1) {{$\gamma$}};
	
	\node (gamma1) at (-2,3) {{$\gamma_1$}};
	\node (gamma2) at (-1,3) {{$\gamma_2$}};
	\node (rho1) at (0,3) {{$\rho_1$}};
	\node (rho2) at (1,3) {{$\rho_2$}};
	\node (rho3) at (2,3) {{$\rho_3$}};
	\node (dots) at (3,3) {{$\cdots$}};
	
	\path (0,0) edge[-] (gamma);
	%
	\path (gamma) edge[-] (gamma1);
	\path (gamma) edge[-] (gamma2);
	\path (gamma) edge[-] (rho1);
	\path (gamma) edge[-] (rho2);
	\path (gamma) edge[-] (rho3);
	%
	%
	\node (rhoe) at (-1.5,4) {{Simulated by $\mc{L}_0$}};
	%
	\draw [decorate,decoration={brace,amplitude=10pt}]
	(-2.3,3.3) -- (-0.7,3.3);
	\end{tikzpicture}
\end{center}
Now $\mc{L}_1$ guesses at the outcome of $\mc{L}_0$, and if for example $\mc{L}_0$ has the finitary outcome $\rho_1$, and $\mc{L}_1$ guesses this, then $\mc{L}_1$ must simulate computations extending $\rho_1$.
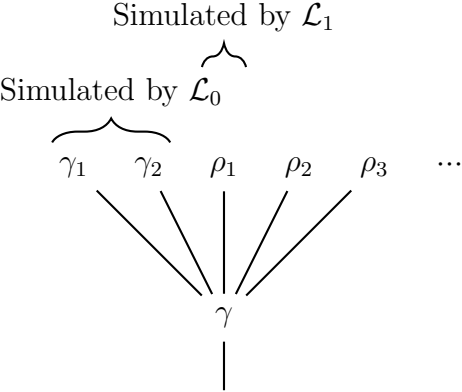
\begin{center}
	\begin{tikzpicture}[-,>=stealth',shorten >=1pt,shorten <=1pt, auto,node
	distance=2cm,thick,every loop/.style={<-,shorten <=1pt}]

	\node (gamma) at (0,1) {{$\gamma$}};
	
	\node (gamma1) at (-2,3) {{$\gamma_1$}};
	\node (gamma2) at (-1,3) {{$\gamma_2$}};
	\node (rho1) at (0,3) {{$\rho_1$}};
	\node (rho2) at (1,3) {{$\rho_2$}};
	\node (rho3) at (2,3) {{$\rho_3$}};
	\node (dots) at (3,3) {{$\cdots$}};
	
	\path (0,0) edge[-] (gamma);
	%
	\path (gamma) edge[-] (gamma1);
	\path (gamma) edge[-] (gamma2);
	\path (gamma) edge[-] (rho1);
	\path (gamma) edge[-] (rho2);
	\path (gamma) edge[-] (rho3);
	%
	%
	\node (rhoe) at (-1.5,4) {{Simulated by $\mc{L}_0$}};
	%
	\draw [decorate,decoration={brace,amplitude=10pt}]
	(-2.3,3.3) -- (-0.7,3.3);
	
	\node (rhoe) at (0,5) {{Simulated by $\mc{L}_1$}};
	%
	\draw [decorate,decoration={brace,amplitude=10pt}]
	(-0.3,4.3) -- (0.3,4.3);
	\end{tikzpicture}
\end{center}
Now in the next step we found extensions $\rho_1^*,\rho_1^{**}$ of $\rho_1$ that split with the other extensions. Before, we could simply extend $\rho_1$ to $\rho_1^*$ and $\rho_1^{**}$. Now, while we are looking for the extensions, $\mc{L}_1$ might simulate other computations, say $\tau_1,\tau_2,\ldots$, extending $\rho_1$. We cannot remove these from the tree. So as before, $\mc{L}_1$ stops simulating them:
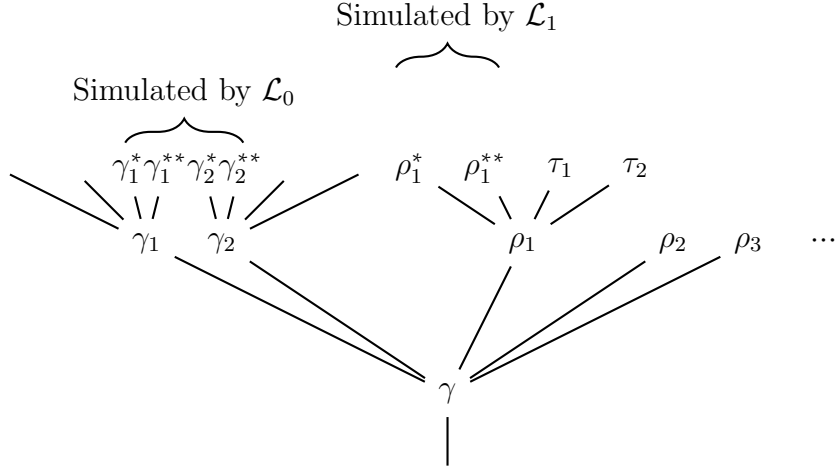
\begin{center}
	\begin{tikzpicture}[-,>=stealth',shorten >=1pt,shorten <=1pt, auto,node
	distance=2cm,thick,every loop/.style={<-,shorten <=1pt}]

	\node (gamma) at (0,1) {{$\gamma$}};
	
	\node (tau1) at (-5,4) {{}};
	\node (tau2) at (-6,4) {{}};
	\node (tau3) at (-1,4) {{}};	
	\node (tau4) at (-2,4) {{}};
	
	\node (gamma1) at (-4,3) {{$\gamma_1$}};
	\node (gamma2) at (-3,3) {{$\gamma_2$}};
	\node (gamma1s) at (-4.25,4) {{$\gamma_1^*$}};
	\node (gamma2s) at (-2.75,4) {{$\gamma_2^{**}$}};
	\node (gamma1ss) at (-3.75,4) {{$\gamma_1^{**}$}};
	\node (gamma2ss) at (-3.25,4) {{$\gamma_2^{*}$}};
	\node (rho1) at (1,3) {{$\rho_1$}};
	\node (rho1s) at (-0.5,4) {{$\rho_1^*$}};
	\node (rho1ss) at (0.5,4) {{$\rho_1^{**}$}};
	\node (rho1st) at (1.5,4) {{$\tau_1$}};
	\node (rho1st2) at (2.5,4) {{$\tau_2$}};
	\node (rho2) at (3,3) {{$\rho_2$}};
	\node (rho3) at (4,3) {{$\rho_3$}};
	\node (dots) at (5,3) {{$\cdots$}};
	
	\path (0,0) edge[-] (gamma);
	%
	\path (gamma) edge[-] (gamma1);
	\path (gamma) edge[-] (gamma2);
	\path (gamma1) edge[-] (gamma1s);
	\path (gamma1) edge[-] (gamma1ss);
	\path (gamma1) edge[-] (tau1);
	\path (gamma1) edge[-] (tau2);
	\path (gamma2) edge[-] (tau3);
	\path (gamma2) edge[-] (tau4);
	\path (gamma2) edge[-] (gamma2s);
	\path (gamma2) edge[-] (gamma2ss);
	\path (gamma) edge[-] (rho1);
	\path (gamma) edge[-] (rho2);
	\path (gamma) edge[-] (rho3);
	\path (rho1) edge[-] (rho1s);
	\path (rho1) edge[-] (rho1ss);
	\path (rho1) edge[-] (rho1st);
	\path (rho1) edge[-] (rho1st2);
	%
	%
	\node (rhoe) at (-3.5,5) {{Simulated by $\mc{L}_0$}};
	%
	\draw [decorate,decoration={brace,amplitude=10pt}]
	(-4.3,4.3) -- (-2.7,4.3);
	
	\node (rhoe) at (0,6) {{Simulated by $\mc{L}_1$}};
	%
	(-4.3,5.3) -- (-2.7,5.3);
	
	\draw [decorate,decoration={brace,amplitude=10pt}]
	(-0.7,5.3) -- (0.7,5.3);
	\end{tikzpicture}
\end{center}
Now we have arrived to the point where computations above $\tau_1$ and $\tau_2$ are no longer being simulated by any $\mc{L}$ requirements, and when we extend $\gamma_1^*$, $\gamma_1^{**}$, $\gamma_2^*$, and $\gamma_2^{**}$ we can find extensions of $\tau_1$ and $\tau_2$ which $e$-split with these extensions (as well as with extensions of $\rho_2^*$, $\rho_3^*$, etc.).

It was important here that we were only dealing with finitely many lowness requirements at a time, because eventually we arrived at a point where parts of the tree were no longer being simulated by any lowness requirement. In the full construction, there will be some important bookkeeping to manage which lowness requirements are being considered at any particular time, so that we sufficiently delay the introduction of new lowness requirements. (This will be accomplished by giving each element of the tree a \textit{scope} in the next section.)

\medskip

This is not the only case where one $\mc{L}$ requirement needs to simulate computations through nodes not simulated by another computation. We will introduce a relation $\mathcal{L}_{d_1}^{\nu_1} \rightsquigarrow \mc{L}_{d_2}^{\nu_2}$ which means that $\mathcal{L}_{d_1}^{\nu_1}$ must simulate computations above nodes that are kept on the tree to be the finitary outcome of $\mc{L}_{d_2}^{\nu_2}$, but which are not simulated by $\mc{L}_{d_2}^{\nu_2}$. We suggest reading $\mathcal{L}_{d_1}^{\nu_1} \rightsquigarrow \mc{L}_{d_2}^{\nu_2}$ as ``$\mathcal{L}_{d_1}^{\nu_1}$ \textit{watches} $\mc{L}_{d_2}^{\nu_2}$''. Given the previous example, if $d_1 < d_2$ and $\nu_1 \prec \nu_2$ then we will have $\mathcal{L}_{d_1}^{\nu_1} \rightsquigarrow \mc{L}_{d_2}^{\nu_2}$, but there will be other, more complicated, cases where $\mathcal{L}_{d_1}^{\nu_1} \rightsquigarrow \mc{L}_{d_2}^{\nu_2}$.

\medskip

Now consider the case of two $\mathcal{M}$ requirements $\mathcal{M}_0$ and $\mathcal{M}_1$ which are of higher priority than two $\mc{L}$ requirements $\mc{L}_0$ and $\mc{L}_1$. Suppose that $\mathcal{M}_0$ successfully builds a $0$-splitting tree $T_0$, and suppose that $\mathcal{L}_0$ and $\mc{L}_1$ both correctly guesses this. Then $\mc{M}_1$ is trying to build a $1$-splitting subtree of $T_0$. Suppose that:
\begin{enumerate}
	\item $\mathcal{L}_0$ guesses that $\mathcal{M}_1$ builds a 1-splitting subtree of $T_0$. If $\mathcal{M}_1$ succeeds at building a $1$-splitting subtree $T_1$ of $T_0$, $\mc{L}_0$ must be able to simulate computations on $T_1$. On the other hand, when $\mc{M}_0$ built $T_0$, there were some nodes of $T_0$ which $\mc{L}_0$ did not simulate. So to make sure that $\mc{L}_0$ simulates computations through $T_1$, $\mc{M}_1$ should look for 1-splits through nodes of $T_0$ that are simulated by $\mathcal{L}_0$.
	
	\item $\mathcal{L}_1$ guesses that $\mathcal{M}_1$ fails to build a 1-splitting tree. If this guess is correct, then there is some node $\sigma$ above which $\mathcal{M}_1$ fails to find a 1-split. $\mathcal{M}_1$ defines a subtree $T_1'$ of $T_0$ containing no 1-splits. The tree $T_1'$ is not being defined dynamically---it is just the subtree of $T_0$ above $\sigma$ where $\mathcal{M}_1$ looked for (and failed to find) a 1-split. 
	
	Recall that $\mathcal{M}_0$ acted specifically to keep simulations computed by $\mathcal{L}_1$ on the tree $T_0$. The tree $T_1'$ also needs to keep these computations. Putting everything together, this means that we should be looking for 1-splits through the nodes which were kept on $T_0$ for the sake of $\mathcal{L}_1$; these are nodes which are not simulated by $\mathcal{L}_1$, but which might be used for the finitary outcomes of $\mc{L}_1$.
\end{enumerate}
So we see from (1) that $\mc{M}_1$ should look for 1-splits through nodes that are simulated by $\mc{L}_0$, and from (2) that $\mc{M}_1$ should look for 1-splits through the nodes which are kept on $T_0$ for the sake of $\mc{L}_1$, but which are not simulated by $\mc{L}_1$. This suggests that $\mc{L}_0$ should simulate the nodes which are kept on $T_0$ for the sake of $\mc{L}_1$, so that $\mathcal{L}_0 \rightsquigarrow \mc{L}_1$.

\medskip

Given a string of outcomes $\nu$, define $\Delta(\nu) \in \{f,\infty\}^{< \omega}$ to be the string
\[ \la \nu(\mc M_0),\nu(\mc M_1),\nu(\mc M_2),\ldots \ra\]
except that we replace any entry which is not $\infty$ with $f$. We can put an ordering $\precsim$ on these using the lexicographic order with $\infty < f$. Suppose that $\nu_1$ and $\nu_2$ are guesses by $\mathcal{L}_{d_1}$ and $\mathcal{L}_{d_2}$ respectively at the outcomes of the higher priority requirements. Define $\mathcal{L}_{d_1}^{\nu_1} \rightsquigarrow \mathcal{L}_{d_2}^{\nu_2}$ if and only if $\Delta(\nu_1) \prec \Delta(\nu_2)$ in the ordering just defined.

\section{Construction}

\subsection{Procedure for constructing splitting trees}

Given the tree $T_{e-1}$ constructed by an instance of $\mc{M}_{e-1}$, we will describe the (attempted) construction by $\mc{M}_{e}$ of an $e$-splitting subtree $T$. This construction will be successful if there are enough $e$-splittings in $T_{e-1}$. We write $T[n]$ for the tree up to and including the $n$th level.

Let $\xi$ be the guess by the particular instance of $\mc{M}_e$ at the outcomes of lower priority requirements. This guess $\xi$ determines an instance of $\mc{M}_{e-1}$ compatible with the given instance of $\mc{M}_e$, and $\xi$ also includes a guess at the outcome of $\mc{M}_{e-1}$. The tree $T_{e-1}$ inside of which we build $T$ depends on the outcome of $\mc{M}_{e-1}$, i.e., if $\xi$ guesses that $\mc{M}_{e-1}$ builds an $(e-1)$-splitting tree then $T_{e-1}$ is this tree, and if $\xi$ guesses that $\mc{M}_{e-1}$ fails to do so, then $T_{e-1}$ is the tree with no $(e-1)$-splits witnessing this failure. (This will all be made more precise in the next section; for now we simply describe the procedure of building $T$.) If we are successful in building $T$ then $\mc{M}_e$ will have the infinitary outcome. We will also  leave for later the description of the subtree of $T_{e-1}$ that we use for the finitary outcome. In this section, we just define the procedure \texttt{Procedure($e$,$\rho$,$T_{e-1}$)} for building an $e$-splitting tree $T$ with root $\rho$ in $T_{e-1}$. (The procedure for building $T$ will not actually use the guess $\xi$, other than to determine what the tree $T_{e-1}$ is, but it will be helpful to refer to $\xi$ in the discussion.)

\medskip

When building $T$, $\mc{M}_e$ must take into account instances of lowness requirements $\mc{L}^\eta_d$ with $d \geq e$ and $\eta$ extending $\xi \concat \infty$. (Since $T$ is being built under the assumption that $\mc{M}_e$ has the infinitary outcome, it only needs to respect lowness requirements that guess that this is the case.) When considering $\mc{L}_d^\eta$ while building $T$, we will only need to know the guesses by $\mc{L}_d^\eta$ at the outcomes of the $\mc{M}$ requirements $\mc{M}_{e+1},\ldots,\mc{M}_d$ of lower priority than $\mc{M}_e$ but higher priority than $\mc{L}_d$ (because the only lowness requirements we need to consider have the same guesses at the requirements $\mc{M}_1,\ldots,\mc{M}_e$), and moreover we will only care about whether the guess is the infinitary outcome or a finitary outcome. 

The tree $T$ will be a \textit{labeled tree}, as follows. Figure \ref{figone} below may help the reader to understand the structure of the tree. Each node of the tree $T$ will be labeled to show which lowness requirements are simulating it, and which lowness requirements are using it for the finitary outcome (of diagonalizing against a computable set $R$). Each node $\sigma$ is given a \textit{scope} $\scope(\sigma)$ and a \textit{label} $\ell(\sigma)$. The scope is a natural number $\geq 0$, and the label is an element of
\[ \Labels = \{f,\infty\}^{< \omega} \cup \{ \top \}.\]
The scope represents the number of lowness requirements that are being considered at this level of the tree, i.e., if the scope of a node is $n$, then we are considering lowness requirements $\mc{L}_e,\ldots,\mc{L}_{e+n}$. If a node $\sigma$ has scope $n$, then the label of $\sigma$ will be an element of
\[ \Labels_n = \{f,\infty\}^{\leq n} \cup \{ \top \}.\]
Note that the label might have length less than $n$, and might even be the empty string; an element of $\Labels_n$ of length $m$ corresponds to a guess by $\mc{L}_{e+m}$. We order the labels lexicographically with $\infty < f$, and with $\top$ as the greatest element. E.g., in $\Labels_2$, we have
\[\top \succ ff \succ f \infty \succ f \succ \infty f \succ \infty \infty \succ \infty \succ \varnothing.\]
We use $\preceq$ for this ordering. We often think of this ordering as being an ordering $\preceq_n$ on $\Labels_n$, and write $\pred_n(\eta)$ for the predecessor of $\eta$ in $\Labels_n$. Though $\Labels$ is well-founded, it does not have order type $\omega$, and so we need to restrict to $\Labels_n$ to make sense of the predecessor operator.

We think of elements of $\{f,\infty\}^n$ as guesses by $\mc{L}_{e+n}$ at the outcomes of $\mc{M}_{e+1},\ldots,\mc{M}_{e+n}$, and $\top$ is just an element greater than all of the guesses. Given an instance $\mc{L}_{e+n}^\eta$ of a lowness requirement of lower priority than $\mc{M}_e$, we write $\Delta_{> e}(\eta)$ for 
\[ \la \Delta(\eta)(e+1),\Delta(\eta)(e+2),\ldots,\Delta(\eta)(e+n) \ra \in \{f,\infty\}^{n}.\]
Recall that $\Delta(\eta)$ just replaces the entries of $\eta$ by $f$ or $\infty$; $\Delta_{> e}(\eta)$ is the string which consists of the guesses of $\eta$ at the outcomes (finitary $f$ or infinitary $\infty$) of $\mc{M}_{e+1},\ldots,\mc{M}_d$. Given two requirements $\mc{L}_{d_1}^{\nu_1}$ and $\mc{L}_{d_2}^{\nu_2}$ with $\nu_1,\nu_2$ extending $\xi \concat \infty$, $\mc{L}_{d_1}^{\nu_1} \rightsquigarrow \mc{L}_{d_2}^{\nu_2}$ if and only if $\Delta_{> e}(\nu_1) \prec \Delta_{> e}(\nu_2)$ lexicographically.

The label $\ell(\sigma^*)$ means that $\sigma^*$ was kept on the tree in order to preserve simulated computations by instances of lowness requirements $\mc{L}_{e+|\ell(\sigma^*)|}^\eta$ with $\Delta_{>e}(\eta) = \ell(\sigma^*)$. So if $\mc{L}_d^\eta$ is an instance of a lowness requirement, and $\sigma^*$ is the child of $\sigma$ on $T$, with $d \leq e+\scope(\sigma^*)$, then:
\begin{itemize}
	\item if $\Delta_{> e}(\eta) \prec \ell(\sigma^*)$, then if $\mc{L}_d^\eta$ simulates computations through $\sigma$, it also simulates computations through  $\sigma^*$; and	
	\item if $\Delta_{> e}(\eta) = \ell(\sigma^*)$, then $\mc{L}_d^\eta$ does not simulate computations through $\sigma^*$, but $\sigma^*$ might be used for the finitary outcome of $\mc{L}_d^\eta$.
\end{itemize}
A node $\sigma^*$ with $\ell(\sigma^*) = \top$ is simulated by every lowness requirement that simulates its parent. Note that we always say that $\sigma^*$ is simulated if its parent $\sigma$ is simulated, rather than just saying that $\sigma^*$ is simulated. This is because when we apply $\Delta_{>e}$ to the guesses $\eta$ of different instances of $\mc{L}_d^\eta$, we lump together many instances with different values for the finitary outcomes. Some of these instances may not simulate computations through $\sigma^*$ because one of the higher priority requirements forces $A$ to extend a node incompatible with $\sigma$, while other instances may be forced to simulate computations through $\sigma^*$ because of these higher priority requirements. For a particular instance of a lowness requirement, there is some initial segment of $A$ determined by the higher priority requirements; think of the labels $\ell$ as applying above this initial segment.

\medskip{}

Certain levels of the tree $T_e$ will be called \textit{expansionary levels}. From the $n$th expansionary level of the tree on, we will begin to consider requirements $\mc{L}_{e+1},\ldots,\mc{L}_{e+n}$, using guesses from at the outcomes of $\mc{M}_{e+1},\ldots,\mc{M}_{e+n}$. (Recall that the scope of a node represents the lowness requirements that it considers.) The nodes at the $n$th expansionary level or higher, but below the $n+1$st expansionary level, will be said to be in the \textit{$n$th strip}. We write $e_1,e_2,e_3,\ldots$ for the expansionary levels. The expansionary levels are defined statically by $e_1 = 0$ and
\[ e_{i+1} = e_i + 2^{i+5}.\]
One might expect that if $\sigma$ is in the $n$th strip, then $\scope(\sigma)$ will be $n$. This will not quite be the case; an expansionary level is where we start considering more requirements, but this might not happen immediately for particular nodes. Instead, if $\sigma$ is in the $n$th strip, we will have $\scope(\sigma) = n-1$ or $\scope(\sigma) = n$. The scope of a child will always be at least the scope of its parent. We say that $\sigma^*$, a child of $\sigma$, is an \textit{expansionary node} if $\scope(\sigma^*) > \scope(\sigma)$. We say that an expansionary node $\sigma^*$ is an \textit{$n$th expansionary node} if $\scope(\sigma^*) = n$. Along any path in the tree, there is one expansionary node for each $n$; it is clear from the fact that the lengths of labels are non-decreasing along paths that there is at most one, and we will show in Lemma \ref{lem:expansionary} that there is at least one. Moreover, the $n$th expansionary node along a path will occur in the $n$th strip.

Suppose that $\sigma^*$ is a child of $\sigma$, with $\scope(\sigma) = n$. If $\sigma^*$ is an $(n+1)$st expansionary node, then we will have $\scope(\sigma^*) = n+1$ and $\ell(\sigma^*) = \top$. Otherwise, $\scope(\sigma^*) = \scope(\sigma) = n$ and we will have either $\ell(\sigma^*) = \ell(\sigma)$ or $\ell(\sigma^*) \prec \ell(\sigma)$. When the label stays the same (or when $\sigma^*$ is an expansionary node), we say that $\sigma^*$ is a \textit{main child} of $\sigma$. Each node $\sigma$ will have exactly two main children, which will $e$-split with each other. Otherwise, if  $\ell(\sigma^*) \prec \ell(\sigma)$, then we say that $\sigma^*$ is a \textit{secondary child} of $\sigma$.

\begin{sidewaysfigure}[pt]
	\begin{center}
	\begin{tikzpicture}[level 1/.style={sibling distance=8cm},
	level 2/.style={sibling distance=2cm}, 
	level 3/.style={sibling distance=0.7cm},
	level 4/.style={sibling distance=0.5cm},
	level 6/.style={sibling distance=0.9cm},
	level 7/.style={sibling distance=0.7cm},
	level 8/.style={sibling distance=0.7cm},]
	\node {$\top$}
	child { node {$\top$} edge from parent[thick,solid]
		child { node {$\top$} edge from parent[thick,solid]
			child { node {$\top$} edge from parent[thick,solid]}
			child { node {$\top$} edge from parent[thick,solid]}
			child { node {$f$} edge from parent[dashed]}}
		child { node {$\top$} edge from parent[thick,solid]
			child { node {$\top$} edge from parent[thick,solid]}
			child { node {$\top$} edge from parent[thick,solid]
				child { node[draw,circle,color=black,inner sep=1.5pt,solid,thin] {$\top$} edge from parent[thick,solid]}
				child { node[draw,circle,color=black,inner sep=1.5pt,solid,thin] {$\top$} edge from parent[thick,solid]}
				child { node {$f$} edge from parent[dashed]
					child { node[draw,circle,color=black,inner sep=1.5pt,solid,thin] {$\top$} edge from parent[thick,solid]}
					child { node[draw,circle,color=black,inner sep=1.5pt,solid,thin] {$\top$} edge from parent[thick,solid]}
					child { node {$\infty$} edge from parent[dashed]
						child { node[draw,circle,color=black,inner sep=1.5pt,solid,thin] {$\top$} edge from parent[thick,solid]
							child[draw,circle,color=black,inner sep=1.5pt,solid,thin] { node {$\top$} edge from parent[thick,solid]}
							child[draw,circle,color=black,inner sep=1.5pt,solid,thin] { node {$\top$} edge from parent[thick,solid]}
							child { node {$ff$} edge from parent[dashed]}}
						child { node[draw,circle,color=black,inner sep=1.5pt,solid,thin] {$\top$} edge from parent[thick,solid]}
						child { node {$\varnothing$} edge from parent[dashed]
							child { node[draw,circle,color=black,inner sep=1.5pt,solid,thin] {$\top$} edge from parent[thick,solid]}
							child { node[draw,circle,color=black,inner sep=1.5pt,solid,thin] {$\top$} edge from parent[thick,solid]}}}}}
			child { node {$f$} edge from parent[dashed]}}
		child { node {$f$} edge from parent[dashed]
			child { node {$f$} edge from parent[thick,solid]}
			child { node {$f$} edge from parent[thick,solid]}
			child { node {$\infty$} edge from parent[dashed]}}}
	child { node {$\top$} edge from parent[thick,solid]
		child { node {$\top$} edge from parent[thick,solid]
			child { node {$\top$} edge from parent[thick,solid]}
			child { node {$\top$} edge from parent[thick,solid]}
			child { node {$f$} edge from parent[dashed]}}
		child { node {$\top$} edge from parent[thick,solid]
			child { node {$\top$} edge from parent[thick,solid]}
			child { node {$\top$} edge from parent[thick,solid]
				child { node[draw,circle,color=black,inner sep=1.5pt,solid,thin] {$\top$} edge from parent[thick,solid]}
				child { node[draw,circle,color=black,inner sep=1.5pt,solid,thin] {$\top$} edge from parent[thick,solid]
					child { node {$\top$} edge from parent[thick,solid]}
					child { node {$\top$} edge from parent[thick,solid]}
					child { node {$ff$} edge from parent[dashed]
						child { node {$ff$} edge from parent[thick,solid]}
						child { node {$ff$} edge from parent[thick,solid]}
						child { node {$f\infty$} edge from parent[dashed]
							child { node {$f\infty$} edge from parent[thick,solid]}
							child { node {$f\infty$} edge from parent[thick,solid]}
							child { node {$f$} edge from parent[dashed]
								child { node {$f$} edge from parent[thick,solid]}
								child { node {$f$} edge from parent[thick,solid]}
								child { node {$\infty f$} edge from parent[dashed]}}}}}
				child { node {$f$} edge from parent[dashed]}}
			child { node {$f$} edge from parent[dashed]}}
		child { node {$f$} edge from parent[dashed]
			child { node {$f$} edge from parent[thick,solid]}
			child { node {$f$} edge from parent[thick,solid]}
			child { node {$\infty$} edge from parent[dashed]}}}
	child { node {$f$} edge from parent[dashed]
		child { node {$f$} edge from parent[thick,solid]
			child { node {$f$} edge from parent[thick,solid]}
			child { node {$f$} edge from parent[thick,solid]}
			child { node {$\infty$} edge from parent[dashed]}}
		child { node {$f$} edge from parent[thick,solid]
			child { node {$f$} edge from parent[thick,solid]}
			child { node {$f$} edge from parent[thick,solid]}
			child { node {$\infty$} edge from parent[dashed]}}
		child { node {$\infty$} edge from parent[dashed]
			child { node {$\infty$} edge from parent[thick,solid]
				child { node[draw,circle,color=black,inner sep=1.5pt,solid,thin] {$\top$} edge from parent[thick,solid]}
				child { node[draw,circle,color=black,inner sep=1.5pt,solid,thin] {$\top$} edge from parent[thick,solid]}
				child { node {$\varnothing$} edge from parent[dashed]
					child { node[draw,circle,color=black,inner sep=1.5pt,solid,thin] {$\top$} edge from parent[thick,solid]}
					child { node[draw,circle,color=black,inner sep=1.5pt,solid,thin] {$\top$} edge from parent[thick,solid]}}}
			child { node {$\infty$} edge from parent[thick,solid]}
			child { node {$\varnothing$} edge from parent[dashed]
				child { node[draw,circle,color=black,inner sep=1.5pt,solid,thin] {$\top$} edge from parent[thick,solid]}
				child { node[draw,circle,color=black,inner sep=1.5pt,solid,thin] {$\top$} edge from parent[thick,solid]}}}};
		
		\node at (10,-9) {\begin{varwidth}{4cm}\footnotesize Nodes labeled $\varnothing$ have no secondary children.\end{varwidth}};
		
		\node[fill=white] at (-3.5,-1.5) {\begin{varwidth}{4cm}\footnotesize Since we have not yet passed the 2nd expansionary stage, main children are not 2nd expansionary nodes.\end{varwidth}};
		
		\node at (4,-7.5) {\begin{varwidth}{4cm}\footnotesize The scope of these nodes is now 2, so the labels can have length $\leq 2$. Along the secondary children, the labels keep decreasing.\end{varwidth}};
		
		\node at (-10,-8) {\begin{varwidth}{3cm}\footnotesize The expansionary nodes do not all occur at the same level, but we can see that each path eventually has one.\end{varwidth}};
		
		\node at (-9,-12) {\begin{varwidth}{5cm}\footnotesize These children labeled $\top$ are not expansionary because we have not reached the 3rd expansionary level.\end{varwidth}};
		
		\node at (-4,-5.7) {\small \textbf{2nd expansionary level.}};
		
		\draw[thick,dashed] (-12,-5.25) -- (12,-5.25);
	\end{tikzpicture}
	\end{center}
	\caption{An example of what the labeled tree might look like. We draw the main children with a solid line and the secondary children with a dashed line. Expansionary nodes are shown by a circle. To fit the tree onto a single page, we have made some simplifications: (a) we have omitted some nodes from the diagram; (b) we have assumed that each node has only one secondary child; and (c) we have assumed that the $2$nd expansionary level $e_2$ is much smaller than the value of 128 that we set in the construction. We show the 2nd expansionary level with the long horizontal dashed line.}\label{figone}
\end{sidewaysfigure}
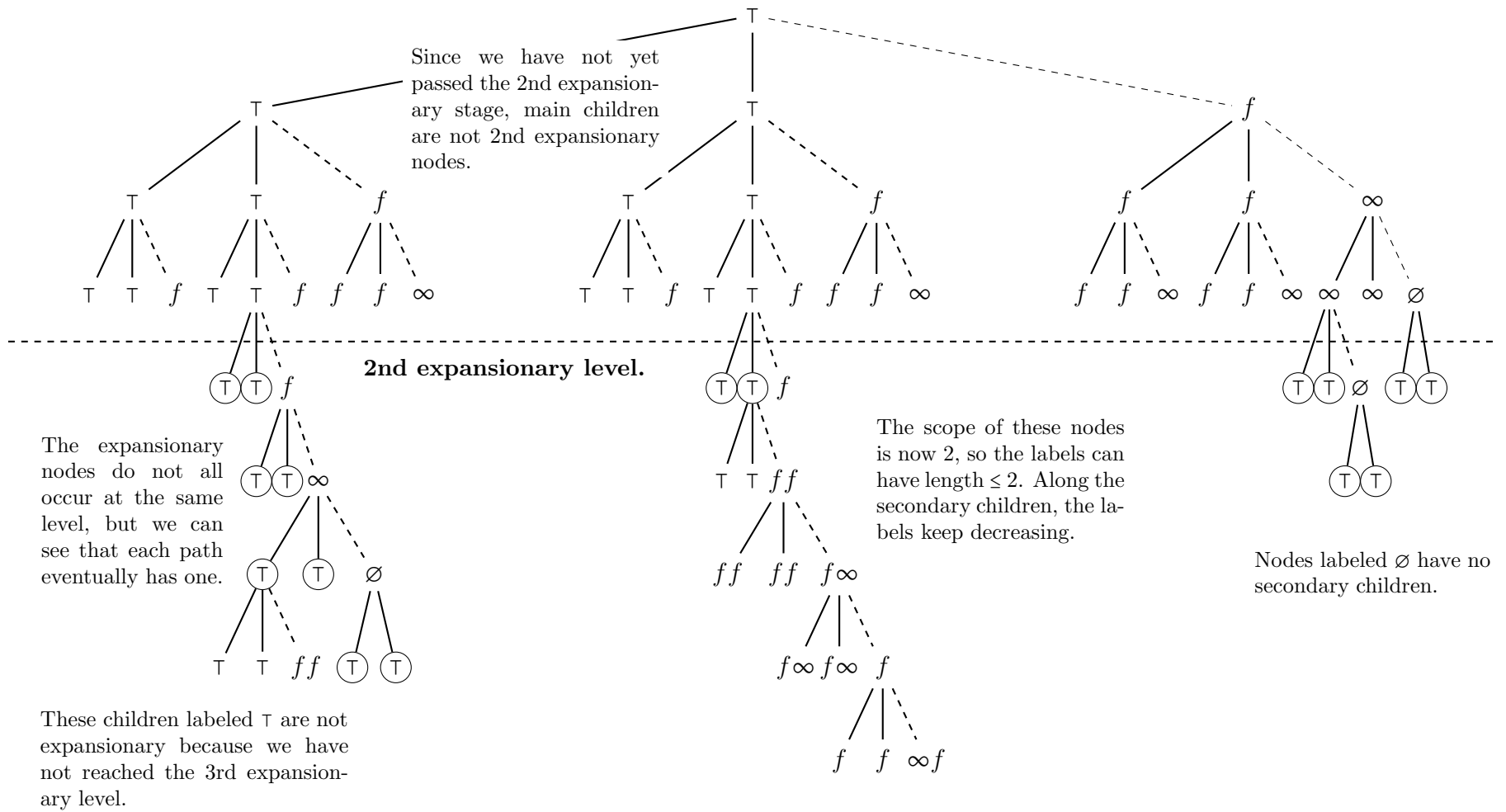

\medskip{}

Recall from the end of the previous section that when we look for splits, we do so through only a subtree of $T_{e-1}$. $T_{e-1}$ itself will be a labeled tree, with scopes $\scope_{e-1} \colon T_{e-1} \to \omega$ and labels $\ell_{e-1} \colon T_{e-1} \to \Labels$. Given $\tau$ on $T_{e-1}$, let $T_{e-1}\treeres{\tau}{\succ f}$ be the subtree of $T_{e-1}$ above $\tau$ (so that $\tau$ is the root node of $T_{e-1}\treeres{\tau}{\succ f}$) such that given $\sigma$ on $T_{e-1}\treeres{\tau}{\succ f}$, the children of $\sigma$ in $T_{e-1}\treeres{\tau}{\succ f}$ are the children $\sigma^*$ of $\sigma$ on $T_{e-1}$ with $\ell_{e-1}(\sigma^*) \succ f$. When we look for a splitting extension of $\tau$, we look through $T_{e-1}\treeres{\tau}{\succ f}$.

In general, for any tree $S$, node $\tau \in S$, and relation $R(\sigma^*)$ (or even a relation $R(\sigma^*,\sigma)$ between a node $\sigma^*$ and its parent $\sigma$), we can define the subtree $S \treeres{\tau}{R}$ as the tree with root node $\tau$, and such that whenever $\sigma \in S \treeres{\tau}{R}$, the children $\sigma^*$ of $\sigma$ on $S \treeres{\tau}{R}$ are exactly the children $\sigma^*$ of $\sigma$ on $S$ such that $R(\sigma^*)$ holds (or $R(\sigma^*,\sigma)$). We will use this notation for the trees $S \treeres{\tau}{\succ \eta}$ and $S \treeres{\tau}{\succeq \eta}$ for $\eta \in \Labels$, and $S \treeres{\tau}{main}$ where $main(\sigma^*,\sigma)$ is the relation of being a main child. So for example $T_{e-1} \treeres{\sigma}{main}[s]$ is the tree consisting of main children of main children of... main children of $\sigma$.

\medskip

The input tree $T_{e-1}$ will have similar properties to those described above. We say that a finitely branching tree $T_{e-1}$ with labels $\ell_{e-1}$ and scopes $\scope_{e-1}$ is \textit{admissible} if:
\begin{enumerate}
	\item Each $\sigma \in T_{e-1}$ has two main children $\sigma^*$ and $\sigma^{**}$ with $\ell_{e-1}(\sigma^*) \succeq \ell_{e-1}(\sigma)$ and $\ell_{e-1}(\sigma^{**}) \succeq \ell_{e-1}(\sigma)$.
	\item If $\sigma^*$ is a child of $\sigma$, then $\scope_{e-1}(\sigma^*) \geq \scope_{e-1}(\sigma)$.
	\item For each $n$, each path through $T_{e-1}$ contains a node $\sigma$ with $\ell_{e-1}(\sigma) = \top$ and $\scope_{e-1}(\sigma) \geq n$.
\end{enumerate}

\bigskip

We are now ready to describe the procedure for constructing splitting trees.

\bigskip

\noindent \texttt{Procedure($e$,$\rho$,$T_{e-1}$):}\\
\textit{Input:} A value $e \geq 0$, an admissible labeled tree $T_{e-1}$ with labels $\ell_{e-1}(\cdot)$ and scopes $\scope_{e-1}(\cdot)$, and a node $\rho$ on $T_{e-1}$ with $\ell_{e-1}(\rho) = \top$. \\
\textit{Output:} A possibly partial labeled $e$-splitting tree $T$, built stage-by-stage.

\smallskip

\noindent \textit{Construction.} To begin, the root node of $T$ is $\rho$ with $\scope(\rho) = 1$ and $\ell(\rho) = \top$. This is the $0$th level of the tree, $T[0]$. At each stage of the construction, if we have so far built $T$ up to the $n$th level $T[n]$, we try to add an additional $n+1$st level.

At stage $s$, suppose that we have defined the tree up to and including level $n$, and the last expansionary level $\leq n$ was $e_t$. Look for a length $l$ such that for each leaf $\sigma$ of $T[n]$, there is an extension $\sigma'$ of $\sigma$ with $\sigma' \in T_{e-1} \treeres{\sigma}{main}$ with $\ell_{e-1}(\sigma') = \top$, and there are extensions $\sigma^*$ and $\sigma^{**}$ of $\sigma$ on $T_{e-1} \treeres{\sigma'}{\succ  f}$, such that $\sigma^*$ and $\sigma^{**}$ are of length $l$, and such that all of these extensions pairwise $e$-split, i.e., for each pair of leaves $\sigma,\tau$ of $T[n]$, these extensions $\sigma^*$, $\sigma^{**}$, $\tau^*$, and $\tau^{**}$ all $e$-split with each other. (At stage $s$, we look among the first $s$-many extensions of these leaves, and we run computations looking for $e$-splits up to stage $s$. If we do not find such extensions, move on to stage $s+1$.)

If we do find such extensions, we will define $T[n+1]$ as follows. To begin, we must wait for $T_{e-1}[s]$ to be defined. In the meantime, we designate each $\sigma$ as \textit{waiting with main children $\sigma^*$ and $\sigma^{**}$}. (This designation is purely for the use of the simulations for lowness requirements, and has no effect on the resulting tree $T$.) While waiting, we still count through stages of the construction, so that after we resume the next stage of the construction will not be stage $s+1$ but some other stage $t > s$ depending on how long we wait. Once $T_{e-1}[s]$ has been defined, for each leaf $\sigma$ of $T[n]$, the children of $\sigma$ in $T[n+1]$ will be:
\begin{itemize}
	\item $\sigma^*$, with:
	\begin{itemize}
		\item if no predecessor of $\sigma$ on $T$ is $t$-expansionary, set $\scope({\sigma}^*) = \scope(\sigma)+1$ and $\ell({\sigma}^*) = \top$, or
		\item $\scope({\sigma}^*) = \scope(\sigma)$ and $\ell({\sigma}^*) = \ell(\sigma)$ otherwise;
	\end{itemize}
	\item $\sigma^{**}$, with:
	\begin{itemize}
		\item $\scope({\sigma}^{**}) = \scope(\sigma)+1$ and $\ell({\sigma}^{**}) = \top$ if no predecessor of $\sigma$ on $T$ is $t$-expansionary, or
		\item $\scope({\sigma}^{**}) = \scope(\sigma)$ and $\ell({\sigma}^{**}) = \ell(\sigma)$ otherwise;
	\end{itemize}
	\item If $\ell(\sigma) \succ \varnothing$, each other extension $\sigma^\dagger$ of $\sigma$ on $T_{e-1} \treeres{\sigma}{\succ \varnothing}[s]$ which is incompatible with $\sigma^*$ and $\sigma^{**}$ will be a child of $\sigma$ on $T$. Put $\scope(\sigma^\dagger) = \scope(\sigma)$. Define $\ell(\sigma^\dagger)$ as follows. Let $n = \scope(\sigma)$. Let $\eta \in \Labels_{n}$ be greatest such that $\sigma^\dagger \in T_{e-1} \treeres{\sigma}{\succeq \eta}$. Then:
	\begin{itemize}
		\item If $\eta$ is $\top$ or begins with $f$, then let $\ell(\sigma^\dagger) = \pred_{n}(\ell(\sigma))$.
		\item  If $\eta$ begins with $\infty$, say $\eta = \infty \eta^*$, then $\ell(\sigma^\dagger)$ will be the minimum, in $\Labels_{n}$, of $\pred_{n}(\ell(\sigma))$ and $\eta^*$.
	\end{itemize} 
	Note that $\pred_n(\ell(\sigma))$ exists because $\ell(\sigma) \succ \varnothing$.
\end{itemize}
The children ${\sigma}^*$ and ${\sigma}^{**}$ are the main children of $\sigma$, and the $\sigma^\dagger$, if they exist, are secondary children. This ends the construction at stage $s$.

\smallskip

\noindent \textit{End construction.}

\bigskip

We say that the procedure is \textit{successful} if it never gets stuck, and construct the $n$th level of the tree $T$ for every $n$. The next lemma is the formal statement that if $T_{e-1}$ has enough $e$-splits, then the procedure is successful.

\begin{lemma}\label{lem:split}
	Fix $e$, an admissible labeled tree $T_{e-1}$, and $\rho \in T_{e-1}$ with $\ell_{e-1}(\rho) = \top$. Suppose that for all $\sigma \in T_{e-1} \treeres{\rho}{\succ \varnothing}$ with $\ell_{e-1}(\sigma) = \top$,
	\begin{itemize}
		\item for all $n$, there is $\tau \in {T}_{e-1}\treeres{\sigma}{\succ f}$ such that $\Phi_e^\tau(n) \downarrow$, and
		\item there are $n$ and $\tau_1,\tau_2 \in {T}_{e-1}\treeres{\sigma}{\succ f}$ such that 
		\[ \Phi_e^{\tau_1} (n) \neq \Phi_e^{\tau_2}(n).\]
	\end{itemize}
	Then {\upshape\texttt{Procedure($e$,$\rho$,$T_{e-1}$)}} is successful.
\end{lemma}

As part of proving this lemma, we will use the following remark, which follows easily from the construction:
\begin{remark}
	Every node on $T$ is a node on $T_{e-1} \treeres{\rho}{\succ \varnothing}$.
\end{remark}

\begin{proof}[Proof of Lemma \ref{lem:split}]
	If we have built $T$ up to level $n$, and $T[n]$ has leaves $\sigma_1,\ldots,\sigma_k$, then as $T_{e-1}$ is admissible, for each $i$ there are $\sigma_i'$ on $T_{e-1} \treeres{\sigma_i}{main}$ with $\ell(\sigma_i') = \top$. By the remark, each $\sigma_i' \in T_{e-1} \treeres{\rho}{\succ \varnothing}$. Then using the assumption of the lemma and standard arguments there are $\sigma_i^{*},\sigma_i^{**}$ on $T_{e-1} \treeres{\sigma_i'}{\succ f}$ such that all of the $\sigma_i^*$ and $\sigma_i^{**}$ pairwise $e$-split. For sufficiently large stages $s$, we will find these extensions.
\end{proof}

The remaining lemmas of this section give properties of the tree constructed by the procedure. The next few lemmas show that the tree $T$ has expansionary levels and is $e$-splitting. As a result, we will see that $T$ is admissible.

\begin{lemma}\label{lem:expansionary}
	Suppose that {\upshape\texttt{Procedure($e$,$\rho$,$T_{e-1}$)}} successfully constructs $T$. For each $\sigma^* \in T[e_{n+1}]$, there is a predecessor $\sigma$ of $\sigma^*$ which is $n$-expansionary.
\end{lemma}
\begin{proof}
	Let $\sigma_0 \in T[e_{n}]$ be the predecessor of $\sigma^*$ at the $n$th expansionary level, and let $\sigma_0,\sigma_1,\sigma_2,\ldots,\sigma_k = \sigma^*$ be the sequence of predecessors of $\sigma^*$ between $\sigma_0$ and $\sigma^*$. If any $\sigma_{i+1}$ were a main child of $\sigma_i$, then either $\sigma_{i+1}$ would be $n$-expansionary, or $\sigma_i$ or one of its predecessors would be $n$-expansionary as desired. If none of these are expansionary, then we must have
	\[ \top \succeq \ell(\sigma_0) \succ \ell(\sigma_1) \succ \ell(\sigma_2) \succ \cdots \succ \ell(\sigma_k) = \ell(\sigma^*) \]
	with all of these in $\Labels_{n-1}$. Since $e_{n+1} > e_n + 2^{n+1} > |\Labels_{n-1}|$, this cannot be the case, and so some predecessor of $\sigma^*$ must be expansionary.
\end{proof}

The following lemma is easy to see by inspecting the construction.

\begin{lemma}\label{lemma:splits-or-down}
	Suppose that {\upshape\texttt{Procedure($e$,$\rho$,$T_{e-1}$)}} successfully constructs $T$. Given distinct leaves $\sigma$ and $\sigma$ of $T[n]$, and $\sigma^*,\tau^* \in T[n+1]$ which are children of $\sigma$ and $\tau$ respectively, either:
	\begin{itemize}
		\item $\sigma^*$ and $\tau^*$ are main children of $\sigma$ and $\tau$ respectively, and $\sigma^*$ and $\tau^*$ $e$-split,
		\item $\sigma^*$ is a secondary child of $\sigma$, and $\scope(\sigma^*) = \scope(\sigma)$ and $\ell(\sigma^*) \prec \ell(\sigma)$, or 
		\item $\tau^*$ is a secondary child of $\tau$, and $\scope(\tau^*) = \scope(\tau)$ and $\ell(\tau^*) \prec \ell(\tau)$.
	\end{itemize}
\end{lemma}

\begin{lemma}
	Suppose that {\upshape\texttt{Procedure($e$,$\rho$,$T_{e-1}$)}} successfully constructs $T$. Given distinct $\sigma$ and $\tau$ in $T$ at the $n$th expansionary level of the tree, and $\sigma^*,\tau^*$ which are extensions of $\sigma$ and $\tau$ respectively at the $n+1$st expansionary level of the tree, $\sigma^*$ and $\tau^*$ are $e$-splitting.
\end{lemma}
\begin{proof}
	Let $\sigma_0 = \sigma,\sigma_1,\sigma_2,\ldots,\sigma_k = \sigma^*$ be the sequence of predecessors of $\sigma^*$ between $\sigma$ and $\sigma^*$, and similarly for $\tau_0 = \tau,\tau_1,\tau_2,\ldots,\tau_k = \tau^*$. Since $\sigma_0$ and $\tau_0$ are at the level $e_n$, and $\sigma^*$ and $\tau^*$ are at the level $e_{n+1}$, we have $k \geq 2^{n+5}$. If, for any $i$, both $\sigma_{i+1}$ and $\tau_{i+1}$ are main children of $\sigma_i$ and $\tau_i$, then by Lemma \ref{lemma:splits-or-down}, $\sigma_{i+1}$ and $\tau_{i+1}$ are $e$-splitting. We argue that this must happen for some $i < k$.
	
	For each $i$, either (a) $\sigma_{i+1}$ is $n$-expansionary and $\ell(\sigma_{i+1}) = \top$, or $\scope(\sigma_{i+1}) = \scope(\sigma_i)$ and either (b) $\ell(\sigma_{i+1}) \prec \ell(\sigma_{i})$, or (c) $\ell(\sigma_{i+1}) = \ell(\sigma_{i})$. There is at most one $i$ for which (a) is the case. Thus there are at most $|\Labels_{n}|+|\Labels_{n+1}| \leq 2^{n+3}$ values of $i$ for which (b) is the case. The same is true for the $\tau_i$. So, as $k \geq 2^{n+5}$, there must be some $i$ for which neither (a) nor (b) is the case for either the $\sigma_i$ or the $\tau_i$. For this $i$, we have both $\sigma_{i+1}$ and $\tau_{i+1}$ are main children of $\sigma_i$ and $\tau_i$ respectively, and so $\sigma_{i+1}$ and $\tau_{i+1}$ are $e$-splitting. Thus $\sigma^*$ and $\tau^*$ are $e$-splitting.
\end{proof}

\begin{lemma}
	Suppose that {\upshape\texttt{Procedure($e$,$\rho$,$T_{e-1}$)}} successfully constructs $T$. $T$ is an $e$-splitting tree: any two paths in $T$ are $e$-splitting.
\end{lemma}
\begin{proof}
	Choose $\sigma_1$ and $\sigma_2$ initial segments of the two paths, long enough that they are distinct, which are at the $n$th expansionary level $T[e_n]$. Let $\tau_1,\tau_2$ be the longer initial segments of the paths at the $n+1$st expansionary level $T[e_{n+1}]$. Then by the previous lemma, $\tau_1$ and $\tau_2$ are $e$-splitting, and so the two paths are $e$-splitting.
\end{proof}

The next lemmas relate the labels of $T$ to the labels on $T_{e-1}$. If a node is labeled on $T_{e-1}$ so that it is not simulated by some lowness requirement, then it should also be labeled on $T$ to not be simulated by that lowness requirement. (The converse is not necessary; $T$ might determine that some node should not be simulated even if that was not determined by $T_{e-1}$.)

\begin{lemma}\label{lem:n-to-n-1}
Suppose that {\upshape\texttt{Procedure($e$,$\rho$,$T_{e-1}$)}} successfully constructs $T$. Given $\sigma \in T$ and $\sigma^* \in T \treeres{\sigma}{\succ \eta}$, with $\ell_{e-1}(\sigma) = \top$, we have $\sigma^* \in T_{e-1} \treeres{\sigma}{\succ \infty \eta}$. In particular, if $\ell_e(\sigma^*) \succ \eta$, then $\ell_{e-1}(\sigma^*) \succ \infty\eta$.
\end{lemma}
\begin{proof}
	It suffices to prove the lemma when $\sigma^*$ is a child of $\sigma$ on $T$, and $\ell(\sigma^*) \succ \eta$. We have two cases, depending on whether $\sigma^*$ is a main child or secondary child of $\sigma$ on $T$.
	\begin{itemize}
		\item If $\sigma^*$ is a main child of $\sigma$ on $T$, then $\sigma^* \in T_{e-1}\treeres{\sigma'}{\succ f}$ and $\sigma' \in T_{e-1}\treeres{\sigma}{main}$, and $\ell_{e-1}(\sigma') = \top$.
		
		Since $\ell_{e-1}(\sigma) = \top$, for every $\tau$ on $T_{e-1}$ between $\sigma$ and $\sigma'$ we have $\ell_{e-1}(\tau) = \top$. Thus $\ell_{e-1}(\sigma') = \top$.
		
		Now for every $\tau$ on $T_{e-1}$ between $\sigma'$ and $\sigma^*$, we have $\ell_{e-1}(\tau) \succ f \succ \infty \eta$. So $\ell_{e-1}(\sigma^*) \succ \infty \eta$.
		
		
		\item If $\sigma^*$ is a secondary child of $\sigma$, let $n = \scope(\sigma)$ and let $\nu \in \Labels_{n}$ be least such that $\sigma^* \in T_{e-1} \treeres{\sigma}{\succeq \nu}$. If $\nu$ is $\top$ or begins with $f$, then $\nu \succ \infty \eta$ and so $\sigma^* \in T_{e-1} \treeres{\sigma}{\succ \infty\eta}$. Otherwise, if $\nu$ begins with $\infty$, say $\nu = \infty \nu^*$, then $\ell(\sigma^*) \succ \eta$ is the minimum, in $\Labels_{n}$, of $\pred_{n}(\ell(\sigma))$ and $\nu^*$. Thus $\nu^* \succ \eta$, which means that $\infty \nu^* \succ \infty \eta$, and so $\sigma^* \in T_{e-1} \treeres{\sigma}{\succ \infty\eta}$.
	\end{itemize}
	This proves the lemma.
\end{proof}

Similarly, we can prove the same lemma but replacing $\succ$ with $\succeq$. We have:

\begin{lemma}\label{lem:n-to-n-1-eq}
	Suppose that {\upshape\texttt{Procedure($e$,$\rho$,$T_{e-1}$)}} successfully constructs $T$. Given $\sigma \in T$ and $\sigma^* \in T \treeres{\sigma}{\succeq \eta}$, we have $\sigma^* \in T_{e-1} \treeres{\sigma}{\succeq \infty \eta}$. In particular, if $\ell_e(\sigma^*) \succeq \eta$, then $\ell_{e-1}(\sigma^*) \succeq \infty\eta$.
\end{lemma}

Finally, putting together results from all of these lemmas, we have:

\begin{lemma}\label{lem:admissible}
		Suppose that {\upshape\texttt{Procedure($e$,$\rho$,$T_{e-1}$)}} successfully constructs $T$. Then $T$ is an admissible tree.
\end{lemma}

\subsection{Minimality requirements}\label{sec:min-req}

In the previous subsection, we described a procedure for constructing an $e$-splitting tree. In this section, we will show how the procedure is applied.

We begin with $\mathbb{T}_{-1}$ defined as in Section \ref{sec:two}. We put $\ell_{-1}(\sigma) = \top$ and $\scope_{-1}(\sigma) = |\sigma|$ for each $\sigma \in \mathbb{T}_{-1}$. This $\mathbb{T}_{-1}$ is admissible. For each instance $\mc{M}_e^\xi$  we will define trees $\mathbb{T}_e^{\xi \concat \infty}$ and $\mathbb{T}_e^{\xi\concat \sigma}$, where $\mathbb{T}_e^{\xi \concat \infty}$ is the outcome of the attempt to construct an $e$-splitting tree, and the $\mathbb{T}_e^{\xi\concat \sigma}$ are subtrees which would witness the failure of the construction of the $e$-splitting tree. Define
\begin{itemize}
	\item $\mathbb{T}_e^{\xi \concat \infty}$ is the labeled tree $T$ produced by \texttt{Procedure($e$,$\sigma$,$\mathbb{T}^{\xi \; \res_{3e-3}}_{e-1}$)} where $\sigma = \xi(\mc{P}_{e-1})$ (or $\sigma$ is the root of $\mathbb{T}_{-1}$ if $e = 0$). The tree $T$ built by this procedure has labels $\ell_e$ and scopes $\scope_e$ defined in its construction.
	\item $\mathbb{T}_e^{\xi \concat \sigma}$ is the tree $\mathbb{T}^{\xi \; \res_{3e-3}}_{e-1}\treeres{\sigma}{\succ f}$. The labels $\ell_e$ of the tree $\mathbb{T}_e^{\xi \concat \sigma}$ are defined by setting:
	\begin{itemize}
		\item $\scope_e(\sigma) = 1$ and $\ell_e(\sigma) = \top$;
		\item for $\tau \neq \sigma$, $\scope_e(\tau) = \scope_{e-1}(\tau)-1$ and $\ell_e(\tau) = \top$ if $\ell_{e-1}(\tau) = \top$, or $\ell_e(\tau) = \eta$ if $\ell_{e-1}(\tau) = f\eta$.
	\end{itemize}
\end{itemize}
There is a uniform and computable construction of all of these trees simultaneously. Whenever in \texttt{Procedure($e$,$\sigma$,$\mathbb{T}^{\xi \; \res_{3e-3}}_{e-1}$)} we need to determine the next level of the tree $\mathbb{T}^{\xi \; \res_{3e-3}}_{e-1}$, the procedure waits until this next level is defined. If $\mathbb{T}^{\xi \; \res_{3e-3}}_{e-1}$ is a total tree then at some point enough of it will be define for the procedure to continue, and it if is not a total tree, then the procedure will get stuck and $\mathbb{T}^{\xi \concat \infty}_{e}$ will also be partial. Similarly, to define a level of $\mathbb{T}_e^{\xi \concat \sigma} = \mathbb{T}^{\xi \; \res_{3e-3}}_{e-1}\treeres{\sigma}{\succ f}$, we need to first build some portion of $\mathbb{T}^{\xi \; \res_{3e-3}}_{e-1}$.

To be a bit more precise, there are two parts of \texttt{Procedure}. First, there is the part where we look for extensions $\sigma^*$ and $\sigma^{**}$ of each leaf $\sigma$ of $T[n]$; and second, after waiting for $T_{e-1}[s]$ to be defined, we define the next level $T[n+1]$ of $T$. If in the first part of \texttt{Procedure}, the tree $T_{e-1}$ has not been sufficiently defined, we just end the current stage of the procedure and restart at the next stage. In the second part of \texttt{Procedure}, we just wait for $T_{e-1}[s]$ to be defined, and then continue from where we were; but while we wait, we still count through the stages, so that when we return it is not at stage $s+1$ but at some greater stage depending on how long we waited for $T_{e-1}$. This will be important to make sure that the simulations used by the lowness requirements are not too slow.

After we define the true path, we will use Lemma \ref{lem:split} to show that along the true path the trees are all fully defined and admissible. When we are not on the true path, e.g.\ if the outcome of $\mc{M}^\eta_e$ is finitary (which means that \texttt{Procedure} is unsuccessful because it cannot find enough $e$-splits), the tree $\mathbb{T}_e^{\eta \concat \infty}$ will be partial, and then any tree defined using this tree will also be partial. Of course, these will all be off the true path.

\subsection{Construction of $A$ and the true path}

We simultaneously define $A$ and a true path of outcomes $\pi$ by finite extension. The construction of the trees in the previous section was uniformly computable, but $A$ and the true path $\pi$ will non-computable. (This of this part of the construction as analogous to choosing a generic in a forcing construction.)

Begin with $A_{-1} = \varnothing$ and $\pi_{-1} = \varnothing$. Suppose that we have so far defined $A_s \prec A$ and $\pi_s \prec \pi$, with $|\pi_s| = s+1$. To define $A_{s+1}$ and $\pi_{s+1}$ we first ask the next requirement what it's outcome is, and then define the extensions appropriately. We use $\mathbb{T}_e$ for the tree defined along the true outcome,  i.e.\ $\mathbb{T}_e := \mathbb{T}_e^{\pi_{3e}}$. Begin with $\pi_{-1} = \varnothing$.
\begin{description}
	\item[$s+1 = 3e$:] Consider $\mathcal{M}_e^{\pi_s}$. By Lemma \ref{lem:split}, either $\mathbb{T}_e^{\pi_s \concat \infty}$ is an $e$-splitting tree, or there is $\sigma \in \mathbb{T}_{e-1} \treeres{\pi_s (\mc{P}_{e-1})}{\succ \varnothing}$ with $\ell_{e-1}(\sigma) = \top$ such that either:
	\begin{enumerate}
		\item there is $n$ such that for all $\tau \in \mathbb{T}\treeres{\sigma}{\succ f}_{e-1}$, $\Phi_e^\tau(n) \uparrow$, or
		\item for all $\tau_1,\tau_2 \in \mathbb{T}\treeres{\sigma}{\succ f}_{e-1}$ and $n$,
		\[ \Phi_e^{\tau_1}(n) \downarrow \ \wedge \ \Phi_e^{\tau_2}(n) \downarrow \quad \longrightarrow \quad  \Phi_e^{\pi_1}(n) = \Phi_e^{\pi_2}(n).\]
	\end{enumerate}
	In the former case, let $\pi_{s+1} = \pi_s \concat \infty$ and $A_{s+1} = A_s$.
	
	In the latter two case, let $\pi_{s+1} = \pi_s \concat \sigma$ and $A_{s+1} = \sigma \succeq A_s$.
	
	\item[$s+1 = 3e + 1$:] Consider $\mathcal{L}_e^{\pi_s}$ with $e = \la e_1,e_2 \ra$. If there is $\sigma \in \mathbb{T}_e$ and $n$ such that $\Psi_{e_1}^\sigma(n) \neq R_{e_2}(n)$, then let $A_{s+1} = \sigma$ and $\pi_{s+1} = \pi_s \concat \sigma$. We may choose $\sigma$ to have $\ell_e(\sigma) = \top$, as (by Lemma \ref{lem:all-trees-admissible}) $\mathbb{T}_e$ is admissible. Otherwise, if there is no such $\sigma$, let $A_{s+1} = A_s$ and $\pi_{s+1} = \pi_s \concat \infty$.
	
	\item[$s+1 = 3e + 2$:] Consider $\mathcal{P}_e^{\pi_s}$. If $A_s = \sigma \in \mathbb{T}_e$, let $\tau_1$ and $\tau_2$ be the two main children of $\sigma$ on $\mathbb{T}_e$. Choose $A_{s+1} \succ \sigma$ to be whichever of $\tau_1,\tau_2$ is not an initial segment of the $e$th c.e.\ set $W_e$.
\end{description}
We define $A = \bigcup_s A_s$ and the true sequence of outcomes $\pi = \bigcup_s \pi_s$. We denote by $\pi_{\mc{R}}$ the true outcome up to and including the requirement $\mc{R}$; for example, $\pi_{\mc{M}_e} = \pi_{3e}$; and similarly for $A_{\mc{R}}$.

In the following lemma, we prove that along the true path, the trees that we construct are total and admissible.

\begin{lemma}\label{lem:all-trees-admissible}
	For each $e$, $\mathbb{T}_e$ is an admissible labeled tree.
\end{lemma}
\begin{proof}
	We argue by induction on $e$. $\mathbb{T}_{-1}$ is an admissible tree. Given $\mathbb{T}_e$ total and admissible, if $\mc{M}_{e+1}$ has the infinitary outcome then $\mathbb{T}_{e+1}$ is defined from $\mathbb{T}_e$ using the \texttt{Procedure}, which is successful, and hence by Lemma \ref{lem:admissible} is admissible.
	
	So suppose that $\mc{M}_{e+1}$ has the finitary outcome, and $\mathbb{T}_{e+1}$ is the tree $\mathbb{T}_e \treeres{\sigma}{\succ f}$ for some $\sigma \in \mathbb{T}_{e-1} \treeres{\pi_s (\mc{P}_{e-1})}{\succ \varnothing}$ with $\ell_{e-1}(\sigma) = \top$.	The labels $\ell_e$ of the tree $\mathbb{T}_e^{\xi \concat \sigma}$ are defined by setting:
	\begin{itemize}
		\item $\scope_e(\sigma) = 1$ and $\ell_e(\sigma) = \top$;
		\item for $\tau \neq \sigma$, $\scope_e(\tau) = \scope_{e-1}(\tau)-1$ and $\ell_e(\tau) = \top$ if $\ell_{e-1}(\tau) = \top$, or $\ell_e(\tau) = \eta$ if $\ell_{e-1}(\tau) = f\eta$.
	\end{itemize}
	We must argue that $\mathbb{T}\treeres{\sigma}{\succ f}$ is admissible:
	\begin{enumerate}
		\item Each $\sigma \in \mathbb{T}_{e+1}$ has two main children $\sigma^*$ and $\sigma^{**}$, namely the same two main children of $\sigma$ in $\mathbb{T}_e$.  We have $\ell_{e+1}(\sigma^*) = \ell_{e}(\sigma^*) - 1 \succeq \ell_{e}(\sigma) - 1 = \ell_{e+1}(\sigma)$ and similarly for $\sigma^{**}$.
		\item If $\sigma^*$ is a child of $\sigma$ on $\mathbb{T}_{e+1}$, then $\sigma^*$ is a child of $\sigma$ on $\mathbb{T}_e$ and $\scope_{e+1}(\sigma^*) = \scope_{e}(\sigma) - 1 \geq \scope_{e}(\sigma) - 1 = \scope_{e+1}(\sigma)$.
		\item For each $n$, each path through $\mathbb{T}_{e+1}$ is also a path through $\mathbb{T}_e$, and hence contains a node $\sigma$ with $\ell_{e+1}(\sigma) = \ell_e(\sigma) = \top$ and $\scope_{e+1}(\sigma) = \scope_{e}(\sigma) - 1 \geq n$.\qedhere
	\end{enumerate}
\end{proof}

\section{Verification}

In this section, we check that the $A$ constructed above is non-computable, of minimal degree, and low for speed.

\subsection{Non-computable}

We chose the initial segment $A_{3e+2}$ of $A$ such that it differs from the $e$th c.e.\ set. Thus $A$ is not computable.

\subsection{Minimal degree}

Recall that we write $\mathbb{T}_e$ for $\mathbb{T}_e^{\pi_{\mc{M}_e}}$, the tree produced by the true outcome of $\mc{M}_e$, and we sometimes write $\mc{M}_e$ for the instance of $\mc{M}_e$ acting along the true sequence of outcomes. We show that $A$ is of minimal degree by showing that it lies on the trees $\mathbb{T}_e$ which are either $e$-splitting or force $\Phi_e^A$ to be partial or computable.

\begin{lemma}\label{lem:A-on-tree}
	For all $e$, $A \in [\mathbb{T}_e]$.
\end{lemma}
\begin{proof}
	Note that for each $e$, $A_{3e-1}$ is the outcome of $\mathcal{P}_{e-1}$ and so it is the root node of $\mathbb{T}_e$. Then we choose $A_{3e} \preceq A_{3e+1} \preceq A_{3e+2}$ in $\mathbb{T}_e$. Since for each $e' \geq e$, $A_{e'} \in \mathbb{T}_{e'} \subseteq \mathbb{T}_{e}$, the lemma follows.
\end{proof}

\begin{lemma}
	$A$ is a minimal degree.
\end{lemma}
\begin{proof}
	$A$ is non-computable. We must show that $A$ is minimal. Suppose that $\Phi_e^A$ is total. If the outcome of $\mathcal{M}_e^{\pi_{3e}}$ is $\infty$, then $A$ lies on the $e$-splitting tree $\mathbb{T}_e = T$ produced by \texttt{Procedure($e$,$A_{\mc{P}_{e-1}}$,$\mathbb{T}_{e-1}$)} and hence $\Phi_e^A \geq_T A$. If the outcome of $\mathcal{M}_e^{\pi_{3e}}$ is $\sigma$, then $A$ lies on $\mathbb{T}_e = \mathbb{T}_{e-1}\treeres{\sigma}{\succ f}$ and (since $\Phi_e^A$ is total) for all $\tau_1,\tau_2 \in \mathbb{T}_{e-1}\treeres{\sigma}{\succ f}$ and for all $n$,
	\[ \Phi_e^{\tau_1}(n) \downarrow \ \wedge \ \Phi_e^{\tau_2}(n) \downarrow \quad \longrightarrow \quad  \Phi_e^{\tau_1}(n) = \Phi_e^{\tau_2}(n).\]
	Thus $\Phi_e^A$ is computable.
\end{proof}

\subsection{Low-for-speed}

Our final task is to show that $A$ is low-for-speed. Because we now have to deal with running times, we need to be a bit more precise about the construction of the trees $\mathbb{T}$ in Section \ref{sec:min-req}. Fixing a particular set of parameters for \texttt{Procedure($e$,$\rho$,$T_{e-1}$)}, one can check that the $s$th stage takes time polynomial in $s$ and $e$. (If the leaves of $T$ have been designated \textit{waiting}, then we charge the time required to wait for $T_{e-1}[s]$ to be defined to stage $s+1,s+2,\ldots$.) In checking this, it is important to note that because all of these trees are subtrees of $\mathbb{T}_{-1}$, there are only polynomially in $s$ many elements of each tree of length (as a binary string) at most $s$. Thus by dovetailing all of the procedures, we can ensure that the $s$th stage of each instance of \texttt{Procedure} takes time polynomial in $s$; the particular polynomial will depend on the parameters for \texttt{Procedure}.

As we build all of the trees $\mathbb{T}$, we keep track of them in an easy-to-query way (such as using pointers) so, e.g., querying whether an element is in a tree can be done in time polynomial in the length (as a binary string) of that element. Again, we use the fact that all of these elements are in $\mathbb{T}_{-1}$.

\medskip

Now we will define the simulation procedure. Fix a lowness requirement $\mc{L}_e$. Define $\eta = \Delta(\pi_{\mc{L}}) \in \{f,\infty\}^{e+1}$; $\eta$ is the sequence of guesses, $f$ or $\infty$, at the outcomes of $\mc{M}_0,\ldots,\mc{M}_e$. Write $\eta_{> i}$ for the final segment $\la \eta(i+1),\ldots,\eta(e) \ra$ of $\eta$, the guesses at $\mc{M}_{i+1},\ldots,\mc{M}_e$.

Write $\scope_i$ and $\ell_{i}$ for the scope and labeling function on $\mathbb{T}_i$. Let $\rho_1,\ldots,\rho_k$ be incomparable nodes on $\mathbb{T}_e$ such that (a) every path on $\mathbb{T}_e$ passes through some $\rho_i$, (b) each $\rho_i$ has $\ell_e(\rho_i) = \top$, and (c) for each $i$ and each $e' \leq e$, $\scope_{e'}(\rho_i) \geq e-e'$. We can find such $\rho_i$ because $\mathbb{T}_e$ is admissible. (Think of the $\rho_i$ as an open cover of $\mathbb{T}_e$ by nodes whose scope, in every $\mathbb{T}_{e'}$ ($e' \leq e$), includes $\mc{L}_e$.) For each $i$, we define a simulation $\Xi_{e,i}$ which works for extensions of $\rho_i$. (It will be non-uniform to know which $\Xi_{e,i}$ to use to simulate $A$, as we will need to know which $\rho_i$ is extended by $A$.) Fix $i$, for which we define the simulation $\Xi = \Xi_{e,i}$:

\medskip

\noindent \textit{Simulation $\Xi = \Xi_{e,i}$:} Begin at stage $0$ with $\Xi(x) \uparrow$ for all $x$. At stage $s$ of the simulation, for each $\sigma \in \mathbb{T}_{-1}$ with $|\sigma| < s$ and $\sigma \succeq \rho_i$, check whether, for each $e' \leq e$, if $\mathbb{T}_{e'}[n]$ is the greatest level of the tree $\mathbb{T}_{e'}$ defined by stage $s$ of the construction of the trees $\mathbb{T}$, then either:
\begin{itemize}
	\item $\sigma$ is on $\mathbb{T}_{e'}[n]$ and $\sigma \in \mathbb{T}_{e'} \treeres{\rho_i}{\succ \eta_{>e'}}$, or
	\item $\sigma$ extends a leaf $\sigma'$ of $\mathbb{T}_{e'}[n]$, with $\sigma' \in \mathbb{T}_{e'} \treeres{\rho_i}{\succ \eta_{>e'}}$, and:
	\begin{itemize}
		\item[($*$)] if $\mathbb{T}_{e'}$ is defined using \texttt{Procedure} ($\mc{M}_{e'}$ has the infinitary outcome), \[ \pred_{\scope_{e'}(\sigma')}(\ell_{e'}(\sigma')) = \eta_{> e'},\]
		and $\sigma'$ has at stage $s$ been designated \textit{waiting with main children $\sigma^*$ and $\sigma^{**}$}, then $\sigma$ extends or is extended by either $\sigma^*$ or $\sigma^{**}$.
	\end{itemize} 
\end{itemize}
If for some $\sigma$ this is true for all $e' \leq e$, then for any $k < s$ with $\Psi^{\sigma}_s(k) \downarrow$, set $\Xi(k) = \Psi^{\sigma}_s(k)$ if it is not already defined.

\medskip

The idea behind the condition ($*$) is that if $\sigma'$ has been designated \textit{waiting}, this is a warning that the secondary children of $\sigma$ will not be simulated by any lowness requirement with guess $\preceq \eta_{> e'}$. So, if $\mc{L}_{e}$ is such a lowness requirement, and if $\sigma$ is along a secondary child of $\sigma'$, then we should not simulate $\sigma$.

\medskip

When we say stage $s$ of the construction of the tree $\mathbb{T}_{e'}$, we mean stage $s$ in \texttt{Procedure} if that is how $\mathbb{T}_{e'}$ is defined, or if $\mc{M}_{e'}$ has a finitary outcome then that part of $\mathbb{T}_{e'}$ which can be defined from stage $s$ in the construction of $\mathbb{T}_{e'-1}$. Recall that we can run these constructions up to stage $s$ in time polynomial in $s$. Thus:

\begin{remark}\label{rem:sim-poly}
	Stage $s$ of the simulation can be computed in time polynomial in $s$. (The polynomial may depend on $e$.) This is because there are polynomially many in $s$ nodes $\sigma \in \mathbb{T}_{-1}$ with $|\sigma| < s$.
\end{remark}

\medskip

The next series of lemmas are proved in the context above of a fixed $e$, with $\rho_1,\ldots,\rho_k$. If $e = \la e_1,e_2 \ra$, we write $\Psi$ for $\Psi_e$. Fix $j$ such that $A$ extends $\rho_j$.

\medskip

If $\mc{L}_e$ has outcome $\infty$, then we need $\Xi_e$ to include the initial segments of $A$ in the computations that it simulates. First we prove that the initial segments of $A$ have the right labels to be simulated.

\begin{lemma}\label{lem:result-watched}
	If $\pi(\mc{L}_e) = \infty$, for each $e' \leq e$, $A \in [\mathbb{T}_{e'} \treeres{\rho_j}{\succ \eta_{>e'}}]$ for some $i$.
\end{lemma}
\begin{proof}
	Let $\sigma = A_{\mc{M}_e}$ be the root of $\mathbb{T}_e$. Since $\pi(\mc{L}_e) = \infty$, $A_{\mc{L}_e} = A_{\mc{M}_e}$. Let $\sigma^* = \xi(\mc{P}_{e}) = A_{\mc{P}_e}$. Then, by construction, $\sigma^*$ a main child of $\sigma$. $A$ is a path through $\mathbb{T}_{e+1}$ extending $\sigma^*$, and $\mathbb{T}_{e+1}$ is one of the following two trees, depending on the outcome of $\mc{M}_{e+1}$:
	\begin{enumerate}
		\item the labeled tree $T$ produced by \texttt{Procedure($e+1$,$\sigma^*$,$\mathbb{T}_e$)}, which is a subtree of $\mathbb{T}_e \treeres{\sigma^*}{\succ \varnothing}$; or
		\item the tree $\mathbb{T}_{e} \treeres{\tau}{\succ f}$ for some $\tau \in \mathbb{T}_e \treeres{\sigma^*}{ \succ \varnothing}$ with $\ell_e(\tau) = \top$.
	\end{enumerate}
	In either case, $A \in [\mathbb{T}_{e} \treeres{\rho_j}{\succ \varnothing}]$, and $\varnothing = \eta_{> e}$. (Note that $\rho_j$ extends $\sigma$ and is extended by $A$.)
	
	Now we argue backwards by induction. Suppose that $A \in [\mathbb{T}_{e'} \treeres{\rho_j}{\succ \eta_{>e'}}]$. We want to argue that $A \in [\mathbb{T}_{e'-1} \treeres{\rho_j}{\succ \eta_{>e'-1}}]$. We have two cases, depending on the outcome of $\mc{M}_e$:
	\begin{itemize}
		\item The outcome of $\mc{M}_e$ is $\infty$. Then $\eta_{>e'-1} = \infty \eta_{>e'}$ and $\mathbb{T}_{e'}$ is the labeled tree $T$ produced by \texttt{Procedure($e'$,$A_{\mc{P}_{e'-1}}$,$\mathbb{T}_{e'-1}$)}. By Lemma \ref{lem:n-to-n-1}, given $\sigma \in \mathbb{T}_{e'}$ and $\sigma^* \in \mathbb{T}_{e'} \treeres{\sigma}{\succ \eta_{>e'}}$, we have that $\sigma^* \in \mathbb{T}_{e'-1} \treeres {\sigma}{\succ \infty \eta_{>e'}} = \mathbb{T}_{e'-1} \treeres {\sigma}{\succ \eta_{>e'-1}}$. As $A \in [\mathbb{T}_e\treeres{\rho_j}{\succ \eta_{> e'}}]$ we have $A \in [\mathbb{T}_{e'-1} \treeres{\rho_j}{\succ \eta_{>e'-1}}]$. 
		
		\item The outcome of $\mc{M}_e$ is $f$. Then $\eta_{>e'-1} = f \eta_{>e'}$ and $\mathbb{T}_{e'}$ is the tree $\mathbb{T}_{e'-1} \treeres{\tau}{\succ f}$ for some $\tau \in \mathbb{T}_{e-1} \treeres{A_{\mc{L}_{e'-1}}}{\succ \varnothing}$ with $\ell_{e-1}(\tau) = \top$. The labels on $\mathbb{T}_{e'}$ are defined so that if $\sigma \in \mathbb{T}_{e'}$, then $\ell_{e'-1}(\sigma) = f \ell_{e'}(\sigma)$ or $\ell_{e'-1}(\sigma) = \top$. Thus, as $A \in [\mathbb{T}_{e'}\treeres{\rho_j}{\succ \eta_{>e'}}]$, and $\mathbb{T}_{e'} = \mathbb{T}_{e'-1} \treeres{\tau}{\succ f}$, we get that $A \in [\mathbb{T}_{e'-1} \treeres{\rho_j}{\succ f\eta_{>e'}}] = [\mathbb{T}_{e'-1} \treeres{\rho_j}{\succ \eta_{>e'-1}}]$.\qedhere
	\end{itemize}
\end{proof}

Now we prove that if $\Psi^A(r)$ converges, then the simulation $\Xi(r)$ converges as well, though it is possible that it will have a different value if there was some other computation $\Psi^\sigma(r)$ which converged before $\Psi^A(r)$ did. Moreover, the simulation will not be too much delayed.

\begin{lemma}\label{lem:result-watched}
	If $\pi(\mc{L}_e) = \infty$, and $\Psi^A(r) \downarrow$, then $\Xi(r) \downarrow$. Moreover, there is a polynomial $p$ depending only on $e$ such that if $\Psi_s^A(r) \downarrow$, then $\Xi_{p(s)}(r) \downarrow$.
\end{lemma}
\begin{proof}
	By the previous lemma for each $e' \leq e$, $A \in [\mathbb{T}_{e'} \treeres{\rho_j}{\succ \eta_{>e'}}]$. Let $\sigma \in \mathbb{T}_{e}$ be an initial segment of $A$ and $s$ a stage such that $\Psi^\sigma_s(r) \downarrow$. We may assume that $\sigma$ is sufficiently long that $\sigma$ extends $\rho_j$.
	
	Fix $e'$ and let $\mathbb{T}_{e'}[n]$ be the greatest level of the tree $\mathbb{T}_{e'}$ defined by stage $s$. Then, as $A \in [\mathbb{T}_{e'} \treeres{\rho_j}{\succ \eta_{>e'}}]$, $\sigma \in \mathbb{T}_{e'} \treeres{\rho_j}{\succ \eta_{>e'}}$. We check the conditions (for $e'$) from the definition of the simulation $\Xi$. Either $\sigma \in \mathbb{T}_{e'} \treeres{\rho_j}{\succ \eta_{>e'}}[n]$, or some initial segment $\sigma'$ of $\sigma$ is in $\mathbb{T}_{e'} \treeres{\rho_j}{\succ \eta_{>e'}}[n]$. In the second case, let us check that we satisfy ($*$). We only need to check ($*$) in the case that $\mathbb{T}_{e'}$ was defined using \texttt{Procedure},
	\[ \pred_{\scope_{e'}(\sigma')}(\ell_{e'}(\sigma')) = \eta_{> e'},\]
	and $\sigma'$ has at stage $s$ been designated \textit{waiting with main children $\sigma^*$ and $\sigma^{**}$}. Now, if $\sigma \in \mathbb{T}_e$ does not extend $\sigma^*$ or $\sigma^{**}$, then it would extend a secondary child $\sigma^\dagger$ of $\sigma'$ with \[ \ell_{e'}(\sigma^{\dagger}) \preceq \pred_{\scope_{e'}(\sigma')}(\ell_{e'}(\sigma')) = \eta_{> e'}.\]
	This contradicts the fact that $\sigma \in \mathbb{T}_{e'} \treeres{\rho_j}{\succ \eta_{>e'}}$.
		
	Since this is true for every $e' \leq e$, and $\Psi_s^\sigma(r) \downarrow$, the simulation defines $\Xi(r) = \Psi_s^\sigma(r)$ if $\Xi(r)$ is not already defined.
	
	\medskip{}
	
	The simulation defines $\Xi(r)$ at the $s$th stage of the simulation. By Remark \ref{rem:sim-poly}, the $s$th stage of the simulation can be computed in time polynomial in $s$.
\end{proof}

Lemma \ref{lem:result-watched} covers the infinitary outcome of $\mc{L}_e$. For the finitary outcome, we need to see that any computation simulated by $\Xi_e$ is witnessed by a computation on the tree, because the use of such a computation was not removed from the tree.

\begin{lemma}\label{lem:on-tree}
	For all $k$, if $\Xi(r) \downarrow$ then there is $\sigma \in \mathbb{T}_e$, $\sigma \succeq \rho_j$, such that $\Psi^{\sigma}(r) = \Xi(r)$.
\end{lemma}
\begin{proof}
	Since $\Xi(r) \downarrow$, by definition of the simulation, there is a stage $s$ of the simulation and $\sigma \in \mathbb{T}_{-1}$ with $|\sigma| < s$ such that, for each $e' \leq e$, if $\mathbb{T}_{e'}[n]$ is the greatest level of the tree $\mathbb{T}_{e'}$ defined by stage $s$, then either:
	\begin{enumerate}
		\item $\sigma$ is on $\mathbb{T}_{e'}[n]$ and $\sigma \in \mathbb{T}_{e'} \treeres{\rho_j}{\succ \eta_{>e'}}$, or
		\item $\sigma$ extends a leaf $\sigma'$ of $\mathbb{T}_{e'}[n]$, with $\sigma' \in \mathbb{T}_{e'} \treeres{\rho_j}{\succ \eta_{>e'}}$, and:
		\begin{itemize}
			\item[($*$)] if $\mathbb{T}_{e'}$ is defined using \texttt{Procedure} ($\mc{M}_{e'}$ has the infinitary outcome), \[ \pred_{\scope_{e'}(\sigma')}(\ell_{e'}(\sigma')) = \eta_{> e'},\]
			and $\sigma'$ has at stage $s$ been designated \textit{waiting with main children $\sigma^*$ and $\sigma^{**}$}, then $\sigma$ extends or is extended by either $\sigma^*$ or $\sigma^{**}$.
		\end{itemize} 
	\end{enumerate}
	$\Xi(r)$ was defined to be $\Psi^{\sigma}_s(r)$ for some such $\sigma$.

	We argue by induction on $i \leq e$ that there is a $\sigma \in \mathbb{T}_{i}\treeres{\rho_j}{\succeq \eta_{> i}}$, with the parent of $\sigma$ in $\mathbb{T}_{i}\treeres{\rho_j}{\succ \eta_{> i}}$, with $\Xi(r) = \Psi^{\sigma}_s(r)$ and satisfying, for each $e'$ with $i < e' \leq e$, either (1) or (2). By the previous paragraph, this is true for $i = -1$. If we can show it for $i = e$, then the lemma is proved. All that is left is the inductive step. Suppose that it is true for $i$; we will show that it is true for $i+1$. We have two cases, depending on the outcome of $\mc{M}_{i+1}$.
	
	\begin{case}
		$\pi_{\mc{M}_{i+1}} = \infty$.
	\end{case}

	 Since $\pi_{\mc{M}_{i+1}} = \infty$, $\mathbb{T}_{i+1}$ is the $(i+1)$-splitting tree defined by \texttt{Procedure($i+1$,$A_{\mc{P}_{i}}$,$\mathbb{T}_i$)}. Fix $\sigma$ from the induction hypothesis: $\sigma \in \mathbb{T}_{i}\treeres{\rho_j}{\succeq \eta_{> i}}$, the parent of $\sigma$ is in $\mathbb{T}_{i}\treeres{\rho_j}{\succ \eta_{> i}}$, $\Xi(r) = \Psi^{\sigma}_s(r)$ and satisfies, for each $e'$ with $i < e' \leq e$, either (1) or (2). Since $\Psi^{\sigma}_s(r)$ converges, we have that $|\sigma| \leq s$ and so $\sigma \in \mathbb{T}_{i}[s]$. Let $n$ be the greatest level of $\mathbb{T}_{i+1}$ defined by stage $s$.
	
	At stage $s$, either (1) $\sigma$ is already on $\mathbb{T}_{i+1}\treeres{\rho_j}{\succ \eta_{>i+1}}[n]$, in which case we are done, or (2) $\sigma$ extends a leaf $\sigma'$ of $\mathbb{T}_{i+1}[n]$, with $\sigma' \in \mathbb{T}_{i+1} \treeres{\rho_j}{\succ \eta_{>i+1}}$, and:
	\begin{itemize}
		\item[($**$)] if $\sigma'$ has at stage $s$ been designated \textit{waiting with main children $\sigma^*$ and $\sigma^{**}$} and \[ \pred_{\scope_{i+1}(\sigma')}(\ell_{i+1}(\sigma')) = \eta_{> i+1},\]
		then $\sigma$ extends or is extended by either $\sigma^*$ or $\sigma^{**}$.
	\end{itemize} 
	We argue in case (2).
	
	We have $\ell_{i+1}(\sigma') \succ \eta_{> i+1}$. (If $\sigma' = \rho_j$ we use the fact that $\ell_{i+1}(\rho_j) = \top$.) Now at some stage we define the next level of the tree, $\mathbb{T}_{i+1}[n+1]$. When we do this, the children of $\sigma'$ are:
	\begin{itemize}
		\item the main children $\sigma^*,\sigma^{**}$ of $\sigma'$;
		\item each other $\sigma^\dagger \in \mathbb{T}_i \treeres{\sigma'}{\succ \varnothing}[t]$, where $t$ is the stage of the \texttt{Procedure} at which $\sigma'$ was declared \textit{waiting}. We might have $t < s$ if $\sigma'$ was already declared \textit{waiting} before stage $s$. In this case, ($**$) will come into play.
	\end{itemize}
	We divide into two possibilities, and use ($**$) to argue that these are the only two possibilities.
	\begin{itemize}
		\item[(P1)] There is a child $\sigma''$ of $\sigma'$ such that $\sigma''$ extends $\sigma$.
		\item[(P2)] There is a child $\sigma''$ of $\sigma'$, with $\ell_{i+1}(\sigma'') \succ \eta_{> i+1}$, such that $\sigma$ strictly extends $\sigma''$.
	\end{itemize}
	Since $\sigma \in \mathbb{T}_{e'}$, $\sigma$ must be compatible with (i.e., it extends or is extended by) one of the children of $\sigma'$. If we are not in case (P1), then $\sigma$ extends one of the children of $\sigma'$. If $\sigma$ extends one of the main children $\sigma''$ of $\sigma'$, then we have $\ell_{i+1}(\sigma'') \succeq \ell_{i+1}(\sigma') \succ \eta_{> i+1}$ and are in case (P2). So the remaining possibility is that $\sigma$ extends one of the secondary children $\sigma''$ of $\sigma'$. Now $\sigma'' \in \mathbb{T}_i[t]$, where $t$ is the stage of the \texttt{Procedure} at which $\sigma'$ was declared \textit{waiting}. As $\sigma \in \mathbb{T}_i[s]$, because $\sigma''$ does not extend $\sigma$, it must be that $t < s$. Since $\sigma$ properly extends $\sigma''$, and the parent of $\sigma$ in $\mathbb{T}_i$ is in $\mathbb{T}_{i}\treeres{\rho_j}{\succ \eta_{> i}}$, $\sigma'' \in \mathbb{T}_{i}\treeres{\rho_j}{\succ \eta_{> i}}$. Then, looking at the construction of $\mathbb{T}_{i+1}$, $\ell_{i+1}(\sigma'') \succ \eta_{> i+1}$ unless $\pred_{\scope_{i+1}(\sigma')}(\ell_{i+1}(\sigma')) = \eta_{> i+1}$. (Recall that $\ell_{i+1}(\sigma') \succ \eta_{> i+1}$.) But then ($**$) would imply that $\sigma$ extends a secondary child of $\sigma'$, which we assumed was no the case. Thus $\ell_{i+1}(\sigma'') \succ \eta_{> i+1}$. We have successfully argued that (P1) and (P2) are the only possibilities.

	\medskip
	
	We begin by considering (P1). If $\sigma''$ is a main child of $\sigma'$, then we have $\sigma'' \in \mathbb{T}_{i+1} \treeres{\rho_j}{\succ \eta_{> i+1}}$. If it is a secondary child of $\sigma'$, then there may be many possible choices for $\sigma''$. We have $\sigma \in \mathbb{T}_i\treeres{\rho_j}{\succeq \eta_{>i}}$ and so we could have chosen $\sigma'' \in \mathbb{T}_i\treeres{\rho_j}{\succeq \eta_{>i}}$ (for example, we could choose $\sigma''$ to be a main child of the main child of... of $\sigma$ in $\mathbb{T}_i$). Thus $\ell_{i+1}(\sigma'')$ is at least the minimum of $\eta_{>i+1}$ and $\pred_{\scope_{i+1}(\sigma')}(\ell_{i+1}(\sigma')) \succeq \eta_{> i+1}$. So $\sigma'' \in \mathbb{T}_{i+1} \treeres{\rho_j}{\succeq \eta_{> i+1}}$.
	
	Since $\sigma''$ extends $\sigma$, $\Xi(r) = \Psi^{\sigma''}_s(r)$.
	
	Fix $e'$ with $i+1 < e' \leq e$. Now as $\sigma$ was not on $\mathbb{T}_{i+1}[n]$, it cannot be on that part of $\mathbb{T}_{e'}$ constructed by stage $s$. Thus (1) cannot be true for $e'$ and $\sigma$, and so (2) must be true. Then (2) is also true for $e'$ and the extension $\sigma''$ of $\sigma$.
	
	\medskip
	
	Now consider (P2). There is a sequence of children $\sigma_0 = \sigma',\sigma_1 = \sigma'',\sigma_2,\ldots,\sigma_k$, with $\sigma$ strictly extending $\sigma_{k-1}$ and $\sigma_k$ extending $\sigma$, $\sigma_k$ is an initial segment of $A$, and with $k \geq 2$. We first claim that $\sigma_2$ is a main child of $\sigma_1$, $\sigma_3$ is a main child of $\sigma_2$, and so on, up until $\sigma_{k-1}$ is a main child of $\sigma_{k-2}$. Indeed, when we begin to define $(n+2)$nd level of the tree $\mathbb{T}_{i+1}$, we do so at a stage greater than $s$, and so the secondary children of $\sigma_1$ are at a level at least $s$ in $\mathbb{T}_i$ (while $\sigma \in \mathbb{T}_i[s]$). So any secondary child of $\sigma_1$ compatible with $\sigma$ would extend $\sigma$, and hence be $\sigma_k$. A similar argument works for the children of $\sigma_2$, $\sigma_3$, ..., $\sigma_{k-2}$.
	
	Thus $\ell_{i+1}(\sigma_{k-1}) \geq \ell_{i+1}(\sigma_{k-2}) \geq \cdots \geq \ell_{i+1}(\sigma_{1}) \succ \eta_{>i+1}$. So $\sigma_{k-1}$, the parent of $\sigma_k = \sigma$, has $\sigma_{k-1} \in \mathbb{T}_{i+1} \treeres{\rho_j}{\succ \eta_{> i+1}}$.
	
	Now if $\sigma_k$ is a main child of $\sigma_{k-1}$, then we have $\sigma_k \in \mathbb{T}_{i+1} \treeres{\rho_j}{\succ \eta_{> i+1}}$. If it is a secondary child of $\sigma_{k-1}$, then there may be many possible choices for $\sigma_k$. We have $\sigma \in \mathbb{T}_i\treeres{\rho_j}{\succeq \eta_{>i}}$ and so we could have chosen $\sigma_k \in \mathbb{T}_i\treeres{\rho_j}{\succeq \eta_{>i}}$ (for example, we could choose $\sigma_k$ to be a main child of the main child of... of $\sigma$ in $\mathbb{T}_i$). Thus $\ell_{i+1}(\sigma_k)$ is at least the minimum of $\eta_{>i+1}$ and $\pred_{\scope_{i+1}(\sigma_{k-1})}(\ell_{i+1}(\sigma_{k-1})) \succeq \eta_{> i+1}$. So $\sigma_k \in \mathbb{T}_{i+1} \treeres{\rho_j}{\succeq \eta_{> i+1}}$.
	
	Since $\sigma_k$ extends $\sigma$, we have $\Xi(r) = \Psi^{\sigma_k}_s(r)$.

	Fix $e'$ with $i+1 < e' \leq e$. Now as $\sigma$ was not on $\mathbb{T}_{i+1}[n]$, it cannot be on that part of $\mathbb{T}_{e'}$ constructed by stage $s$. Thus (1) cannot be true for $e'$ and $\sigma$, and so (2) must be true. Then (2) is also true for $e'$ and the extension $\sigma_k$ of $\sigma$.

	\begin{case}
		$\pi_{\mc{M}_{i+1}} = \tau$.
	\end{case}
	
	We have that $\sigma \in \mathbb{T}_i \treeres{\rho_j}{\eta_{>i}}$. $\mathbb{T}_{i+1}$ is  the tree $\mathbb{T}_{i} \treeres{\tau}{\succ f}$ for some $\tau \in \mathbb{T}_i \treeres{\pi_{\mc{P}_i}}{\succ \varnothing}$. The labels $\ell_{i+1}$ of the tree $\mathbb{T}_{i+1}$ are defined from the labels $\ell_{i}$ of the tree $\mathbb{T}_i$ by setting $\ell_{i+1}(\sigma) = \top$ if $\ell_{i} = \top$, or $\ell_{i+1}(\sigma) = \eta$ if $\ell_{i} = f\eta$.
	
	Note that since $\pi_{\mc{M}_{i+1}} = \tau$ is the finitary outcome, $\eta_{> i}$ begins with $f$. Thus $\sigma$ is still on $\mathbb{T}_{i+1}$. Moreover, for each $\tau' \in \mathbb{T}_{i+1}$, $\ell_{i}(\tau') \succeq \eta_{>i}$ if and only if $\ell_{i+1}(\tau') \succeq \eta_{i+1}$. So $\sigma$ is on $\mathbb{T}_{i+1}\treeres{\rho_j}{\succ \eta_{>i+1}}$.
\end{proof}

We now show how to use these lemmas to prove that $A$ is low-for-speed.

\begin{lemma}
	$A$ is low-for-speed.
\end{lemma}
\begin{proof}
	Given $\la e,i \ra$, suppose that $\Psi_{e}^A = R_{i}$ and that $\Psi_e^A(n)$ is computable in time $t(n)$. We must show that $R_{i}$ is computable in time $p(t(n))$ for some polynomial $p$. Note that the outcome of $\mathcal{L}_{\la e,i \ra}$ must be $\infty$, as otherwise we would have ensured that $\Psi_{e}^A \neq R_{i}$. Let $j$ be such that $A$ extends the $\rho_j$ from the simulation for $e$. So by Lemma \ref{lem:result-watched} $\Xi_{\la e,i \ra,j}$ is total and there is a polynomial $p$ depending only on $\la e,i \ra$ such that if $\Psi_{e}^A(n)$ is computed in time $s$, then $\Xi_{\la e,i \ra,j}(n)$ is computed in time $p(s)$.
	
	Now we argue that $\Xi_{\la e,i \ra,j}$ computes $R_{i} = \Psi_e^A$. Suppose not; then there is $n$ such that $\Xi_{\la e,i \ra,j}(n) \neq R_{i}(n) = \Psi_{e}^A(n)$.
	Since $\Xi_{\la e,i \ra,j}(n)$ does in fact converge, by Lemma \ref{lem:on-tree} there is $\sigma \in \mathbb{T}_{\la e,i \ra}$ extending $\rho_j$ such that $\Psi_e^\sigma(n) = \Xi_{{\la e,i \ra},j}(n) \neq R_i(n)$.
	This contradicts the fact that the outcome of $\mc{L}_{\la e,i \ra}$ is $\infty$, as we would have chosen $\tau$ as the outcome.
\end{proof}

\bibliography{References}
\bibliographystyle{alpha}

\end{document}